\documentclass[12pt,reqno]{amsart}

\usepackage{amsmath}
\usepackage{amsfonts}
\usepackage{amssymb}
\usepackage{caption}
 \usepackage{latexsym}
\usepackage{color}
\usepackage{hyperref}
\usepackage{enumerate}

\renewcommand\div{\,\big|\,} 
 
  \renewcommand\a{\alpha}  \renewcommand\b{\beta}

 \newcommand\A{\mathrm{A}} \newcommand\AGL{\mathrm{AGL}} \newcommand\AGammaL{\mathrm{A\Gamma L}}  \newcommand\ASL{\mathrm{ASL}} \newcommand\ASigmaL{\mathrm{A \Sigma L}} \newcommand\Aut{\mathrm{Aut}}
 
\def\C{\mathbf{C}}   \newcommand\Cay{\mathrm{Cay}}   \newcommand\Cos{\mathrm{Cos}} 
\newcommand\D{\mathrm{D}}

 \newcommand\GammaL{\mathrm{\Gamma L}}    \newcommand\GL{\mathrm{GL}} \newcommand\GO{\mathrm{GO}} \newcommand\GU{\mathrm{GU}} \newcommand\Ga{{\it \Gamma}}

\newcommand\M{\mathrm{M}}  \renewcommand \mod{{\rm mod~}}
\newcommand\N{\mathbf{N}}
 \renewcommand\O{\mathrm{O}} 
  \newcommand\PGL{\mathrm{PGL}}   \newcommand\PGU{\mathrm{PGU}}  \newcommand\POmega{\mathrm{P\Omega}} \newcommand\PSL{\mathrm{PSL}}    \newcommand\PSp{\mathrm{PSp}} \newcommand\PSU{\mathrm{PSU}}
  \newcommand\PGSp{\mathrm{PGSp}}

 \newcommand\SL{\mathrm{SL}}   \newcommand\Sp{\mathrm{Sp}}    \def\Sy{\mathrm{S}}

\def\soc{ \mathrm{Soc}}

\newcommand\Z{\mathbf{Z}} \newcommand\ZZ{\mathbb{Z}}

\newcommand\mz{\mathbb{Z}}
\newcommand\ID{{\rm Inndiag}}

\newcommand{\cG}{\mathcal{G}}

 \numberwithin {equation} {section}

\newtheorem{theorem}{Theorem}[section]
\newtheorem{lemma}[theorem]{Lemma}

\newtheorem{proposition}[theorem]{Proposition}

\title[ graphs with a quasi-semiregular automorphism]
 { Prime-valent Symmetric graphs with a quasi-semiregular automorphism}
\author{Fu-Gang Yin, Yan-Quan Feng, Jin-Xin Zhou, A-Hui Jia}%
\address{Department of Mathematics, Beijing Jiaotong University, Beijing, 100044, China}
 \email{18118010@bjtu.edu.cn (F.-G. Yin), yqfeng@bjtu.edu.cn (Y.-Q. Feng), \linebreak
  jxzhou@bjtu.edu.cn (J.-X. Zhou),
    18121577@ bjtu.edu.cn (A-Hui Jia)
  } %

  \thanks{Corresponding Author: Yan-Quan Feng.}
 \thanks{This work was supported by the National Natural Science Foundation of China (11731002,12071023) and by the 111 Project of China (B16002).}

\begin{document}
\maketitle

\begin{abstract}
An automorphism of a graph is called {\em quasi-semiregular} if it fixes a unique vertex of the graph and its remaining cycles have the same length. This kind of symmetry of graphs was first investigated by Kutnar, Malni\v{c}, Mart\'{i}nez and Maru\v{s}i\v{c} in 2013, as a generalization of the well-known semiregular automorphism of a graph. Symmetric graphs of valency three or four, admitting a quasi-semiregular automorphism, have been classified in recent two papers.

Let $p\geq 5$ be a prime and $\Ga$ a connected symmetric graph of valency $p$ admitting a quasi-semiregular automorphism. In this paper, we first prove that either $\Ga$ is a connected Cayley graph $\Cay(M,S)$ such that $M$ is a $2$-group admitting a fixed-point-free automorphism of order $p$ with $S$ as an orbit of involutions, or $\Ga$ is a normal $N$-cover of a $T$-arc-transitive graph of valency $p$ admitting a quasi-semiregular automorphism, where $T$ is a non-abelian simple group and $N$ is a nilpotent group. Then in case $p=5$, we give a complete classification of such graphs $\Ga$ such that either $\Aut(\Ga)$ has a solvable arc-transitive subgroup or $\Ga$ is $T$-arc-transitive with $T$ a non-abelian simple group. We also construct the first infinite family of symmetric graphs that have a quasi-semiregular automorphism and an insolvable full automorphism group.
\end{abstract}


\section{Introduction}

Graphs studied in this paper are finite, connected, undirected and simple, and groups are finite.
Let $\Ga$ be a graph with vertex set $V(\Ga)$. An ordered pair of adjacent vertices of $\Ga$ is called an {\em arc}.
Denote by $\Aut(\Ga)$ the full automorphism group of $\Ga$. We say that $\Ga$ is {\em $G$-vertex-transitive} (or {\em $G$-arc-transitive}) if $G\leq\Aut(\Ga)$ and $G$ has a single orbit on $V(\Ga)$ (or the arcs of $\Ga$); we omit the prefix $G$ when $G=\Aut(\Ga)$. In particular, an arc-transitive graph is also called a \emph{symmetric} graph.

Let $G$ be a permutation group on a set $\Omega$. The stabilizer in $G$ of a point $\a\in\Omega$ is
the subgroup $G_\a=\{g\in G \mid \a^g=\a\}$ of $G$. If $G_\a$ is trivial for every $\a\in\Omega$, then we say that $G$ is \emph{semiregular} on $\Omega$, and \emph{regular} if in addition $G$ is transitive on $\Omega$. An element $g$ of $G$ is said to be \emph{semiregular} if $\langle g \rangle$, the subgroup generated by $g$, is semiregular on $\Omega$. Semiregular permutations play an important role in the study of  vertex-transitive graphs. For example, Biggs~\cite{Biggs-graphs}  used semiregular automorphisms to give concise descriptions of some interesting graphs. In 1981, Maru\v{s}i\v{c} asked whether every vertex-transitive digraph has a non-identity semiregular automorphism~\cite[Problem~2.4]{M81}.  Over the last forty years, numerous papers have been published on this question, and we refer the readers to \cite{CSS,DMMN,GV,Li,M2018,M2017,MS,MSV,Ro} for vertex-transitive graphs with certain valencies and orders, to \cite{FengHKKM,GPV,GX,KS,V,Xu} for edge-transitive graphs and to \cite[Section~1]{B-G} and the references therein for a broader survey on these kind of results. Despite consistent effort made, the above Maru\v{s}i\v{c}'s question is still wide open.

Our focus in this paper is on the so-called qusi-semiregular automorphism of a graph, which is a natural generalization of semiregular automorphism and defined below. A permutation group $G$ on a set $\Omega$ is said to be \emph{quasi-semiregular} if $G$ fixes a point, say $\a$, of $\Omega$, and is semiregular on $\Omega\setminus \{\a\}$. A non-identity permutation $g$ on a set $\Omega$ is said to be \emph{quasi-semiregular} if $\langle g \rangle$ is quasi-semiregular on $\Omega$.
As a matter of convenience, we exclude the identity permutation as a quasi-semiregular permutation, which follows from \cite{FengHKKM,YF} but not from \cite{Hu,KMMM}.
Quasi-semiregular permutation groups were first introduced in 2013 by Kutnar, Malni\v{c}, Mart\'{i}nez,  Maru\v{s}i\v{c}~\cite{KMMM}, where they initiated the study of graphs with a quasi-semiregular automorphism group, and later, Hujdurovi$\acute{\rm c}$~\cite{Hu} classified circulants which have a quasi-semiregular automorphism group with at most $5$ orbits. There are a large class of vertex-transitive graphs admitting quasi-semiregular automorphisms which are called GFRs. A \emph{graphical Frobenius representation} (GFR) of a Frobenius group $G$ (a non-regular transitive permutation group with every non-identity permutation fixing at most one point) is a graph whose automorphism group is isomorphic to $G$. Clearly, every GFR has a quasi-semiregular automorphism. In 2018, Doyle, Tucker, Watkins~\cite{GFR} started the study of Frobenius groups admitting a GFR, which has been an active topic of research; see, for example, \cite{KN,Spiga1,Spiga2}. Inspired by the facts listed above, we are naturally led to consider the following problem:

\vspace{0.2cm}
\noindent{\bf Problem A:} Characterize or classify vertex-transitive graphs with a quasi-semiregular automorphism.

The investigation of Problem~A was began by the second author of this paper together with Hujdurovi$\acute{\rm c}$, Kov$\acute{\rm a}$cs, Kutnar and Maru$\check{\rm s}$i$\check{\rm c}$ in \cite{FengHKKM}, where some crucial properties on restrictions to automorphism groups of such graphs were given. In contrast with semiregular automorphism of symmetric graphs, the condition for a symmetric graph having a quasi-semiregular automorphism is quite restrictive. Actually, \cite[Theorem 1.1]{FengHKKM} gave a classification of cubic symmetric graphs with a quasi-semiregular automorphism, and surprisingly, there are only three such graphs: the complete graph $K_4$, the Petersen graph and the Coxeter graph. As for the tetravalent symmetric graph, in a very recent paper~\cite{YF}, the first two authors of this paper used the results in \cite{FengHKKM} to prove that a tetravalent symmetric graph has a quasi-semiregular automorphism if and only if it is a Cayley graph on an abelian group of odd order. In this paper, we shall give a partial solution of Problem~A by characterizing prime valent symmetric graphs admitting a quasi-semiregular automorphism. Before stating our main results, we introduce some terminologies.

Let $G$ be a group. For a subset $S$ of $G$ with $1\not\in S$ and $S=S^{-1}:=\{s^{-1}: s \in S\}$, the \emph{Cayley graph} $\Cay(G,S)$ on $G$ with respect to $S$ is defined to be the graph with vertex set $G$ and edge set $\{\{g,sg\}\ |\ g \in G,s \in S\}$. It is easy to see that $\Cay(G,S)$ is connected if and only if $G=\langle S\rangle$.
Let $\Ga$ be an $G$-vertex-transitive graph and let $N$ be an intransitive normal subgroup of $G$. The {\it normal quotient graph } $\Ga_ N$ of $\Ga$ induced by $N$ is defined to be the graph with vertex set $\{ \alpha^N\ |\ \alpha \in V\Ga \}$, the set of all $N$-orbits on $V\Ga$, and with two orbits $B , C \in \{ \alpha^N\ |\ \alpha \in V\Ga \}$  adjacent in $\Ga_N$ if some vertex in $B$ is adjacent to some vertex in $C$ in $\Ga$. If $\Ga$ and $\Ga_N$ have the same valency, then $\Ga$ is called a \emph{normal $N$-cover} of  $\Ga_N$. A group $G$ is called \emph{almost simple} if the socle $\soc(G)$ of $G$, the product of all minimal normal subgroups of $G$, is a nonabelian simple group. An automorphism of a group is said to be \emph{fixed-point-free} if it fixes only the identity of the group.

We now state our first main theorem.

\begin{theorem}\label{th:valp} Let $p\geq 5$ be a prime and let $\Ga$ be a connected symmetric graph of valency $p$ admitting a quasi-semiregular automorphism. Then one of the following holds:
\begin{enumerate}[\rm (1)]
 \item  If $\Ga$ has a solvable arc-transitive group of automorphisms, then $\Ga\cong\Cay(M,S)$, where $M$ is a $2$-group admitting a fixed-point-free automorphism $\beta$ of order $p$ such that $S$ is a generating set of involutions of $M$ and an orbit of $\langle \beta \rangle$;
\item If every arc-transitive group of automorphisms of $\Ga$ is insolvable, then $\Aut(\Ga)$ has a normal nilpotent subgroup $N$ such that $\Aut(\Ga)/N$ is almost simple, $\Ga$ is a normal $N$-cover of $\Ga_N$, and $\Ga_N$ is a connected $\soc(\Aut(\Ga)/N)$-arc-transitive graph of valency $p$ admitting a quasi-semiregular automorphism.
\end{enumerate}
\end{theorem}

By Theorem~\ref{th:valp}~(1), for a prime $p\geq 5$, to classify $p$-valent graphs that admit a quasi-semiregular automorphism and a solvable arc-transitive group of automorphisms, it is equivalent to classify finite $2$-groups admitting a fixed-point-free automorphism of order $p$ that has an orbit of involutions generating the $2$-group. Finite groups with a fixed-point-free automorphism of prime order have been investigated widely in group theory, and have many very restrictive properties. For example, a finite group with a fixed-point-free automorphism of prime order is nilpotent by Thompson's theorem~\cite{Thompson1959}, and the nilpotency class of the group is bounded by
some function of the prime by Higman \cite{Higman1957}. For more results on finite groups with a fixed-point-free automorphism, we refer to Gorenstein's book \cite[Chapter 10]{DG} and Huppert's book \cite[Chapter X, \S11]{Huppert82b}. 


Theorem~\ref{th:valp}~(2) suggests that, for a prime $p\geq 5$, to classify $p$-valent symmetric graphs that admit a quasi-semiregular automorphism and has no solvable arc-transitive group of automorphisms, we should first determine $T$-arc-transitive $p$-valent graphs with a quasi-semiregular automorphism, where $T$ is a non-abelian simple group.


Our next theorem is a refined form of Theorem~\ref{th:valp} in case $p=5$, and it completely determines pentavalent symmetric graphs $\Ga$ admitting a quasi-semiregular automorphism such that either $\Aut(\Ga)$ has a solvable arc-transitive subgroup or $\Ga$ is $T$-arc-transitive with $T$ a non-abelian simple group.

\begin{theorem}\label{th:val5}
Let $\Ga$ be a connected pentavalent symmetric graph with a quasi-semiregular automorphism and let $v\in V(\Ga)$. Then the  following  statements hold:
\begin{enumerate}[\rm (1)]
 \item $\Ga$ has a solvable arc-transitive group of automorphisms if and only if $\Ga=\cG_{2^4}$, $\cG_{2^8}$ or $\cG_{2^{12}}$;
\item If $\Ga$ is $T$-arc-transitive for a nonabelian simple group $T$, then either $\Ga=K_{6}$, or $\Ga=\cG_{n}$ for $n \in  \{36$, $66$, $126$, $396$, $1456$, $2016$, $22176\}$.
\end{enumerate}

\smallskip
The above graphs $\cG_n$, together with $\Aut(\cG_n)$ and $\Aut(\cG_n)_v$,  are listed in Table~\ref{tb:graphs} and constructed in Eq.~\ref{eq:eq6} and Eq.~~\ref{eq:eq7} of Section~\ref{sec:th3}:

\begin{table}[htb]
\caption{Information for $\cG_n$, $\Aut(\cG_n)$ and $\Aut(\cG_n)_v$} \label{tb:graphs}
\[
 \begin{array}{  l  ll|lll  } \hline
\mbox{Graph} & \Aut(\cG_n) & \Aut(\cG_n)_v & \mbox{Graph} & \Aut(\cG_n) & \Aut(\cG_n)_v   \\ \hline
\cG_{2^4} & \ZZ_2^4\rtimes\Sy_5 &\Sy_5  & \cG_{2^8} & \ZZ_2^8\rtimes \AGL_1(5) & \AGL_1(5) \\
\cG_{2^{12}} & \ZZ_2^{12}\rtimes \AGL_1(5) &  \AGL_1(5)  & \cG_{36}& \A_6.\mz_2^2& \ZZ_2 \times \AGL_1(5) \\
\cG_{66} &\PGL_2(11) &\ZZ_2 \times \D_{10} &
\cG_{126} &\Sy_9 &\Sy_4 \times\Sy_5\\
\cG_{396} &  \M_{11} & \AGL_1(5)&
\cG_{1456} & {}^2B_2(8) & \AGL_1(5)\\
 \cG_{2016} &\PSL_3(4).\mz_2^2 & \ZZ_2 \times \AGL_1(5) &
\cG_{22176} &\M_{22}.\mz_2 &\ZZ_2 \times \AGL_1(5) \\
 \hline
 \end{array}
\]
\end{table}
\end{theorem}

By Theorems~\ref{th:valp} and \ref{th:val5}, to complete the classification of pentavalent symmetric graphs $\Ga$ admitting a quasi-semiregular automorphism, it remains to classify nilpotent normal covers of $K_6$ and one of $\cG_n$'s. However, this seems not easy. We do not expect to finish this task in the present paper, but we obtain the following theorem by proving that for each prime $p>5$ there exists an arc-transitive normal $\mz_p^4$-cover of $K_6$ which has quasi-semiregular automorphism.  Note that there are only two symmetric graphs of valency at most $4$ which admit a quasi-semiregular automorphism and an insolvable full automorphism group.

\begin{theorem}\label{th:infinite}
There are infinitely many connected pentavalent symmetric graphs with a quasi-semiregular automorphism and an insolvable full automorphism group.
\end{theorem}

This paper is organized as follows. After Section~2 with some preliminary results, we prove Theorems~\ref{th:valp}-\ref{th:val5} in Sections~3-4, respectively. To prove Theorem~\ref{th:val5}, we need to determine non-abelian simple groups that contain an element of order $5$ with its centralizer isomorphic to $\ZZ_5$, $\ZZ_{10}$, $\ZZ_{15}$, $\ZZ_{20}$, $\ZZ_5 \times\Sy_4$ or $\ZZ_5 \times \A_4$. This proof is a little long, and we deal with it in Section~5, which depends on the classification of finite non-abelian simple groups. In Section~6, we shall prove Theorem~\ref{th:infinite}, and in Section~7, we give some remarks.

To end this section we fix some notation which is used in this paper. The notation for groups in Table~\ref{tb:graphs} and in this paper is standard; see \cite{Atlas} for example. In particular, denote by $\ZZ_{m}$ and $\D_{m}$ the cyclic and dihedral group of order $m$, and by $\A_m$ and $\Sy_m$ the alternating and symmetric group of degree $m$, respectively. For a prime $p$,  we use $\mz_p^m$ to denote the elementary abelian group of order $p^m$, and also use $\ZZ_{p}$ to denote the finite field of order $p$. For two groups $A$ and $B$, $A\times B$ stands for the direct product of $A$ and $B$, $A.B$ for an extension of $A$ by $B$, and $A\rtimes B$ for a split extension or a semi-direct product of $A$ by $B$.

\section{Preliminary}\label{sec:pre}

Let $\Ga=\Cay(G,S)$ be a Cayley graph on a group $G$ with respect to $S$. For $g\in G$,
the right multiplication $R(g)$ is a permutation on $G$ defined by $h\mapsto hg$ for all $h\in G$. It is easy to see that $R(g)$ is an automorphism of $\Ga$, and $R(G)=\{R(g)\ |\ g \in G \}$ is a regular subgroup of $\Aut(\Ga)$. Furthermore, $\Aut(G,S)=\{\alpha:\alpha \in \Aut(G)\ | \ S^\alpha=S \}$ is also a group of automorphisms of $\Ga$ that fixes the identity in $G$, and $R(G)$ is normalized by $\Aut(G,S)$. Clearly, $R(G)\cap \Aut(G,S)=1$, and so we have  $R(G)\Aut(G,S)=R(G)\rtimes\Aut(G,S)\leq \Aut(\Ga)$. The following proposition is first proved by Godsil~\cite{Godsil}.

\begin{proposition} \label{pro:Cayley}
Let $\Ga=\Cay(G,S)$. Then $N_{\Aut(\Ga)}(R(G))=R(G)\rtimes\Aut(G,S)$.
\end{proposition}

Lorimer~\cite{Lor} proved the following proposition for symmetric graph of prime valency.

\begin{proposition}\label{pro:normalquo}
Let $\Ga $ be a connected $G$-arc-transitive graph of prime valency, with $G\leq \Aut(\Ga)$, and let $N ~ \unlhd~ G$ have at least three orbits on $V\Ga $. Then the following statements hold:
\begin{enumerate}[\rm (1)]
\item  $N$ is semiregular on $V\Ga $, $G/N \leq \Aut(\Ga _N)$, $\Ga _N$ is a connected $G/N$-arc-transitive graph, and $\Ga $ is a normal cover of $\Ga _N$;
\item  $G_v \cong (G/N)_{\Delta}$ for any $v \in V\Ga $ and $\Delta \in V(\Ga_N)$.
\end{enumerate}

\end{proposition}

For a group $G$, we use $\Z(G)$ to denote the center of $G$, and for a subgroup $H$ of $G$, the normalizer and centralizer of $H$  in $G$ are denoted by $\N_G(H)$ and $\C_G(H)$, respectively. For $x \in G$, write $\C_G(x)=\C_G(\langle x\rangle)$, the centralizer of $x$ in $G$. The following  proposition gathers some properties of vertex-transitive graph with a quasi-semiregular automorphism; see \cite[Lemmas 2.1, 2.4 and 2.5]{FengHKKM}.

\begin{proposition} \label{pro:qsprop}
Let $\Ga$ be a connected $G$-vertex-transitive graph with a quasi-semiregular automorphism $x$. Then we have the following statements:
\begin{enumerate}[\rm (1)]
\item $\Ga$ is not bipartite;
\item Let $v$ be the vertex fixed by $x$. Then $\C_G(x)\leq \C_{G_v}(x)$;
\item Let $N \neq 1$ be a normal semiregular subgroup of $G$. Then $N$ is nilpotent, and by conjugation, $x$ is a fixed-point-free automorphism of $N$. Furthermore, if $N$ is intransitive and $\Ga$ is a normal cover of $\Ga_N$, then $xN$ is a quasi-semiregular automorphism of $\Ga_N$.
\end{enumerate}
\end{proposition}

A group with a fixed-point-free automorphism is quite restrictive, and we refer to \cite[Chapter 10: Lemma~1.1]{DG} for the following result.

\begin{proposition} \label{fixed-point-free-atuo}
Let $\b$ be a fixed-point-free automorphism of a group $G$ of order $n$. Then for every $g$ in $G$, we have $gg^{\b}\cdots g^{\b^{n-1}}= g^{\b^{n-1}}\cdots g^{\b}g=1$.
\end{proposition}

The following result is about centralizer of a transitive permutation group, and we refer to \cite[Theorem 4.2A]{DM-book}.

\begin{proposition} \label{regulargroup} Let $G$ be a transitive permutation group on a set $\Omega$. Then the centralizer $C_{S_\Omega}(G)$ of $G$ in the symmetric group $S_{\Omega}$ is semiregular, and if $G$ is abelian then $C_{S_\Omega}(G)=G$ is regular.
\end{proposition}

Feng and Guo~\cite{Feng-Guo} determined vertex stabilizers of pentavalent symmetric graphs. By using {\sc Magma}~\cite{Magma}, we have the following proposition.

\begin{proposition}\label{pro:val5stab}
Let $\Ga$ be a connected pentavalent $G$-arc-transitive graphs with $v \in V\Ga$, and let $x \in G_v$ be of order $5$. Then $ \C_{G_v}(x)$ is solvable, and one of the following holds:

\begin{enumerate}[\rm (1)]
\item  $G_v \in \{ \ZZ_5,\D_{10},\AGL_1(5),\A_5,\Sy_5,\ASL_2(4), \ASigmaL_2(4)\}$ and  $ \C_{G_v}(x)=\ZZ_5$;

\item    $G_v \in \{\ZZ_2 \times \D_{10},\ZZ_2 \times\AGL(1,5) \}$  and $ \C_{G_v}(x)=\ZZ_{10}$;

\item   $G_v \in \{ \AGL_2(4),\AGammaL_2(4) \}$ and $ \C_{G_v}(x)=\ZZ_{15}$;

\item  $G_v = \ZZ_4 \times \AGL_1(5) $ and $ \C_{G_v}(x)=\ZZ_{20}$;

\item   $G_v \in \{ \A_4 \times \A_5,(\A_4 \times \A_5)\rtimes\ZZ_2,
\mz_2^6\rtimes\GammaL_2(4)\}$  and $ \C_{G_v}(x)=\A_4 \times \ZZ_{5}$;

\item  $G_v =\Sy_4 \times\Sy_5 $ and $ \C_{G_v}(x)=\Sy_4 \times \ZZ_{5}$.

\end{enumerate}

\end{proposition}

Let $G$ be a group and $H$ be a core-free subgroup of $G$, that is, $H$ contains no nontrivial normal subgroup of $G$. Let $x \in G$ with $x \notin H$ and $HxH=Hx^{-1}H $. The \emph{coset graph} $\Ga=\Cos(G,H,HxH)$ on $G$ with respect to $H$ and $x$ is defined as the graph with vertex-set $[G{:}H]$, the set of all right cosets of $H$ in $G$, and edge-set $\{\{Hg,Hyg\}\ |\ g\in G, y\in HxH\}$. It is easy to see that $\Ga$ is connected if and only if $G=\langle x,H\rangle$. For every $g\in G$, the permutation $\hat{g}$ on $[G{:}H]$ defined by $Hz\mapsto Hzg$ for all $z\in G$,  is an automorphism of $\Ga$, and $\hat{G}:=\{\hat{g}\ |\ g\in G\}$ is an arc-transitive subgroup of $\Aut(\Ga)$. Since $H$ is core-free, one may easily prove $G\cong \hat{G}$. The following proposition is well-known, and we refer to Sabidussi~\cite{Sabi} for some details.

\begin{proposition}\label{pro:cosetcons}
Let $G$ be a group and $H$ a core-free subgroup of $G$. Let $x \in G$ with $x^2 \in H$ and $\langle H,x \rangle=G$. Then $\Cos(G,H,HxH)$ is $\hat{G}$-arc-transitive of valency $|H{:}(H \cap H^x)|$.

Let $\Ga$ be a connected $G$-arc-transitive graph, and for an edge $\{ v,u \}$ of $\Ga$, let $x \in G$ swap $u$ and $v$. Then $\Ga \cong \Cos(G,G_v,G_v xG_v )$, where $x\in \N_G(G_{vu})$, $x^2 \in G_v$ and $\langle G_v ,x\rangle=G$.

\end{proposition}

It is easy to see $\Cos(G,H,HxH) \cong \Cos(G,H^\alpha,H^\alpha x^\alpha H^\alpha)$ for every $ \alpha \in \Aut(G)$. Given a group $G$ and a core-free subgroup $H$ of $G$, Proposition~\ref{pro:cosetcons} can be applied to find all non-isomorphic $G$-arc-transitive graphs with the vertex stabilizer of $G$ isomorphic to $H$, which can be done by {\sc Magma} when the orders of the graphs are not too large. For {\sc Magma} code, one may refer to \cite[Appendix]{DF19}.

\section{Proof of Theorem \ref{th:valp}}\label{sec:valp}

In this section, we investigate symmetric graphs of prime valency with a quasi-semiregular automorphism. For a positive integer $n$ and a prime $p$, denote by $n_p$ the highest $p$-power dividing $n$. The proof Theorem~\ref{th:valp} follows from three lemmas, of which the first is a simple observation but important.

\begin{lemma}\label{lm:5qs}
Let $p$ be prime and let $\Ga$ be a connected symmetric graph of valency $p$ admitting a quasi-semiregular automorphism. Then the quasi-semiregular automorphism has order $p$, $|V(\Ga)|\equiv 1 \ (\mod p)$,  $|\Aut(\Ga)|_p=p$, and every automorphism of order $p$ of $\Ga$ is quasi-semiregular.
\end{lemma}
\begin{proof} Let $A=\Aut(\Ga)$ and let $\a$ be a quasi-semiregular automorphism of $\Ga$ fixing $v\in V(\Ga)$. Since $\Ga$ has prime valency $p$, $\langle\a\rangle$ is transitive on the neighbours of $v$ in $\Ga$ because all orbits of $\langle\a\rangle$ on the neighbourhood of $v$ have same length, and hence $\a$ has order $p$. It follows that $|V(\Ga)|\equiv 1 \ (\mod p)$. Note that $\Ga$ has valency $p$. Since $\Ga$ is arc-transitive, we have $p\div |A_v|$, and the connectedness of $\Ga$ implies that $p^2\nmid |A_v|$. Thus, $|\Aut(\Ga)|_p=p$, and all Sylow $p$-subgroups of $A$ have order $p$. By the Sylow Theorem, every Sylow $p$-subgroup of $A$ is conjugate to $\langle \a\rangle$, and hence every element of order $p$ of $A$ is quasi-semiregular, as the element of order $p$ generates a Sylow $p$-subgroup of $A$.
\end{proof}

Next we consider the case when a symmetric graph of prime valency $p$ admitting a quasi-semiregular automorphism contains a solvable arc-transitive group of automorphisms.

\begin{lemma}\label{lm:valp1}
 Let $\Ga$ be a connected symmetric graph of prime valency $p\geq 5$. Then $\Aut(\Ga)$ has a solvable arc-transitive subgroup admitting a quasi-semiregular automorphism if and only if $\Ga=\Cay(M,S)$, where $M$ is a $2$-group admitting a fixed-point-free automorphism of order $p$ with $S$ as an orbit and $S$ is a generating set of involutions of $M$.
 \end{lemma}
\begin{proof} To prove the sufficiency, let $M$ be a $2$-group admitting a fixed-point-free automorphism of order $p$, say $\b$, such that $S$ is an orbit of $\langle \b\rangle$ and also $S$ is a generating set of involutions of $M$. Then $|S|=p$ and $\Cay(M,S)$
is a connected graph of valency $p$. Furthermore, $\b\in \Aut(M,S)\leq \Aut(\Cay(M,S))$, implying that $\b$ is an quasi-semiregular automorphism of $\Cay(M,S)$, and $R(M)\rtimes \langle \b\rangle$ is an solvable arc-transitive group of automorphisms of $\Cay(M,S)$.

To prove the necessity, assume that $\Ga$ be a connected symmetric graph of prime valency $p\geq 5$ and that $\Aut(\Ga)$ contains a solvable arc-transitive subgroup, say $G$, admitting a quasi-semiregular automorphism, say $\a$, of order $p$. Let $N$ be a maximal intransitive normal subgroup of $G$, and let $M/N$ be a minimal normal subgroup of $G/N$. Since $G$ is solvable, $G/N$ is solvable and so $M/N$ is elementary abelian.

By Proposition~\ref{pro:qsprop}~(1), $\Ga$ is not bipartite, and hence $N$ has at least three orbits. Then Proposition~\ref{pro:normalquo} implies that $N$ is semiregular and $\Ga_N$ is $G/N$-arc-transitive, where $G/N\leq \Aut(\Ga_N)$. Note that $M/N\unlhd G/N$ implies that $M\unlhd G$. By the maximality of $N$, $M$ is transitive $V(\Ga)$, so $M/N$ is an abelian transitive subgroup of $\Aut(\Ga_N)$. By Proposition~\ref{regulargroup}, $M/N$ is regular on $V(\Ga_N)$, and since $N$ is semiregular on $V(\Ga)$, $M$ is regular on $V(\Ga)$, that is, $\Ga$ is a Cayley graph on $M$. Then we may let $\Ga=\Cay(M,S)$ for some $1\not\in S \subseteq M$ with $S=S^{-1}$, and here we identify the right regular representation $R(M)$ with $M$. Since a conjugacy of $\a$ under every element of $G$ is quasi-semiregular, we may assume $\a\in G_1$, and since $\Ga$ is a connected graph of valency $p$, we have $|S|=p$ and $\langle S\rangle=M$.

Since $M\unlhd G$, by Proposition~\ref{pro:Cayley} we have $G_1\leq \Aut(M,S)$, and hence $\a\in \Aut(M,S)$. Since $\a$ is quasi-semiregular, every orbit of $\a$ on $S$ has the same length, and since $|S|=p$, $\a$ is a fixed-point-free automorphism of order $p$ of $M$ with $S$ as an orbit. Since $p$ is odd and $S=S^{-1}$, $S$ contains an involution, and therefore $S$ consists of involutions because $\langle \a\rangle\leq \Aut(M,S)$ is transitive on $S$.
\end{proof}

Now we consider the case when a symmetric graph of prime valency $p$ admitting a quasi-semiregular automorphism contains no solvable arc-transitive group of automorphisms.

\begin{lemma}\label{lm:valp2}
Let $\Ga$ be a connected symmetric graph of prime valency $p\geq 5$ admitting a quasi-semiregular automorphism, and let $\Aut(\Ga)$ have no solvable arc-transitive subgroup.
Then $\Aut(\Ga)$ has a normal nilpotent subgroup $N$ such that $\soc(\Aut(\Ga)/N)$ is a  nonabelian simple group. Furthermore, the quotient group $\Ga_N$ is a connected $\soc(\Aut(\Ga)/N)$-arc-transitive graph of valency $p$ admitting a quasi-semiregular automorphism.
 \end{lemma}

\begin{proof}
Let $A=\Aut(\Ga)$ and let $\a\in A$ be a quasi-semiregular automorphism. By Lemma~\ref{lm:5qs}, $\a$ has order $p$. Let $N$ be a maximal intransitive normal subgroup of $A$. By Proposition~\ref{pro:qsprop}~(1), $\Ga$ is not bipartite, and $N$ has at least three orbits on $V(\Ga)$. By Proposition~\ref{pro:normalquo}, $N$ is semiregular and $\Ga$ is a normal $N$-cover of $\Ga_N$ with $A/N$ as an arc-transitive group of automorphisms of $\Ga_N$, and by Proposition~\ref{pro:qsprop}~(3), $N$ is nilpotent and $\Ga_N$ has a quasi-semiregular automorphism. Clearly, $\Ga_N$ is connected and has valency $p$. By Lemma~\ref{lm:5qs}, to finish the proof, the only thing left is to show that $A/N$ is almost simple and $\Ga_N$ is $\soc(A/N)$-arc-transitive.

Let $M/N$ be a minimal normal subgroup of $A/N$. Then $M\unlhd A$ and $N<M$. Since $N$ is a maximal intransitive normal subgroup of $A$, $M$ is transitive on $V(\Ga)$ and hence $M/N$ is transitive on $V(\Ga_N)$. If $M/N$ is regular on $V(\Ga_N)$, Proposition~\ref{pro:qsprop}~(3) implies that $M/N$ is nilpotent,  yielding that $M$ is solvable as $N$ is nilpotent. It follows $M\langle \a\rangle$ is a solvable arc-transitive group of automorphisms of $\Ga$, contrary to hypothesis. Therefore, $M/N$ is not regular, and since $A/N$ is arc-transitive and $\Ga$ has prime valency $p$, $M/N$ is arc-transitive. The minimality of $M/N$ in $A/N$ implies that $M/N=T_1\times T_2\times\cdots\times T_m$, where all $T_i$, $1\leq i\leq m$, are isomorphic non-abelian simple groups.

Suppose $m\geq 2$. We first claim that $T_1$ is semiregular on $V(\Ga_N)$. Assume that $T_1$ is transitive on $V(\Ga_N)$. Then for every $1\leq i\leq m$, $T_i$ is transitive, and since $T_1$ commutes with $T_2\times\cdots\times T_m$ elementwise, by Proposition~\ref{regulargroup}, $T_1$ is regular. Now assume that $T_1$ is intransitive. Since $\Ga_N$ contains a quasi-semiregular automorphism, Proposition~\ref{pro:qsprop}~(1) implies that $\Ga_N$ is not bipartite, and hence $T_1$ has at least three orbits on $V(\Ga_N)$ because $T_1\unlhd M/N$ and $M/N$ is arc-transitive on $V(\Ga_N)$. By Proposition~\ref{pro:normalquo}, $T_1$ is semiregular. It follows that in both cases, $T_1$ is always semiregular. By Lemma~\ref{lm:5qs}, $M/N$ contains a quasi-semiregular automorphism of $\Ga_N$, and by Proposition~\ref{pro:qsprop}~(3), $N$ is nilpotent, a contradiction. Thus, $m=1$, that is, $M/N$ is a non-abelian simple group.

Suppose that $A/N$ has another minimal normal subgroup $L/N$ such that $M/N\not=L/N$. Then the above paragraph implies that $L/N$ is  arc-transitive and non-abelian simple. Note that $M/N$ commutes with $L/N$ elementwise. By Proposition~\ref{regulargroup}, both $M/N$ and $L/N$ are regular on $V(\Ga_N)$, a contradiction. Thus, $A/N$ is almost simple, and $\soc(A/N)=M/N$. Since $\Ga_N$ is $M/N$-arc-transitive, it is $\soc(A/N)$-arc-transitive.
\end{proof}

\begin{proof}[Proof of Theorem~\ref{th:valp}] Let $p\geq 5$ be a prime and let $\Ga$ be a connected symmetric graph of valency $p$ admitting a quasi-semiregular automorphism. By Lemma~\ref{lm:5qs}, every arc-transitive group of automorphisms of $\Ga$ has a quasi-semiregular automorphism, and then Theorem~\ref{th:valp}~(1) follows from Lemma~\ref{lm:valp1}, and Theorem~\ref{th:valp}~(2) follows from Lemma~\ref{lm:valp2}.
\end{proof}

\section{Proof of Theorem \ref{th:val5}} \label{sec:th3}

In this section, we prove Theorem \ref{th:val5}. We begin with a simple basic fact that will be frequently used later.

\begin{lemma}\label{lm:p45qs} Let $p\not=5$ be a prime and let $G=\langle a\rangle\times\langle b\rangle\times\langle c\rangle\times\langle d\rangle\cong\mz_p^4$. Then $G$ has an fixed-point-free automorphism $\a$ of order $5$ such that $a^\a=b$, $b^\a=c$, $c^\a=d$ and $d^\a=e$, where $e=a^{-1}b^{-1}c^{-1}d^{-1}$.
\end{lemma}

\begin{proof} Since $G=\langle b\rangle\times\langle c\rangle\times\langle d\rangle\times\langle e\rangle$, where $e=a^{-1}b^{-1}c^{-1}d^{-1}$, we have that the map $a\mapsto b$, $b\mapsto c$, $c\mapsto d$ and $d\mapsto e$, induces an automorphism of $G$, say $\a$. It is easy to check that $e^\a=a$. Since $\langle a,b,c,d,e\rangle=G$, $\a$ has order $5$.

Let $v=a^{i_1}b^{i_2}c^{i_3}d^{i_4} \in G$ with $i_1,i_2,i_3,i_4 \in \ZZ_p$, such that $v^\a=v$. Here $\mz_p$ is viewed as the field of order $p$. Then
  \[ a^{i_1}b^{i_2}c^{i_3}d^{i_3}=(a^{i_1}b^{i_2}c^{i_3}d^{i_4})^{\a}=b^{i_1}c^{i_2}d^{i_3}(abcd)^{-i_4}=a^{-i_4}b^{i_1-i_4}c^{i_2-i_4}d^{i_3-i_4},\]
 and hence
 \[ i_1=-i_4 ,\ i_2 = i_1-i_4,\ i_3 = i_2-i_4,\ i_4 =i_3-i_4.\] It follows
 $5i_4=0$. Since $p\not=5$, we have $i_4=0$, and then $i_1=i_2=i_3=0$. Thus $v=1$, implying that $\a$ is fixed-point-free.
 \end{proof}

Now let us introduce three groups, of which the first is the elementary abelian group $\mz_2^4$:

 \begin{eqnarray*}
&G_{2^4}=&\langle f_1\rangle\times\langle f_2\rangle\times\langle f_3\rangle\times\langle
f_4\rangle\cong \ZZ_2^4, \mbox{ and let } f_0=(f_1f_2f_3f_4)^{-1};  \\
&G_{2^8}=&\langle e_0, e_1,e_2,e_3,e_4,f_0,f_1,f_2,f_3,f_4\ \lvert\
  f_0=(f_1f_2f_3f_4)^{-1}, e_0=(e_1e_2e_3e_4)^{-1}, \\
 && e_i=[f_i,f_{i+1}],\ f_i^2=[e_i,e_j]=[e_i,f_j]=1, \mbox{ for all } i,j \in \ZZ_5 \rangle;\\
 &G_{2^{12}}=&\langle d_i,e_i,f_i: i \in  \ZZ_5 \ \lvert\   f_0=(f_1f_2f_3f_4)^{-1}, e_0=(e_1e_2e_3e_4)^{-1},  d_0=(d_1d_2d_3d_4)^{-1},\\
 & & e_i=[f_i,f_{i+1}], d_i=[e_i,f_{i+2}], f_i^2=[d_i,d_j]=[d_i,e_j]=[d_i,f_j]=1 \text{ \rm for all } i,j \in \ZZ_5 \rangle.
 \end{eqnarray*}

It is easy to see that for each $n\in \{2^4,2^8,2^{12}\}$, $G_n$ is generated by $f_1,f_2,f_3,f_4$.  By Lemma \ref{lm:p45qs}, the map $f_i\mapsto f_{i+1}$ for $i\in\mz_5$, induces a fixed-point-free automorphism of $G_{2^4}$, and this is also true for  $G_{2^8}$ and $G_{2^{12}}$ by using {\sc Magma}~\cite{Magma}. Furthermore, $f_i$ for all $i\in\mz_5$ are involutions of $G_n$ for every $n\in \{2^4,2^8,2^{12}\}$.

\begin{lemma}\label{th:M} Let $G$ be a finite $2$-group generated by four involutions $f_1$, $f_2$, $f_3$ and $f_4$, and let $f_0=(f_1f_2f_3f_4)^{-1}$. Suppose that $G$ has a fixed-point-free automorphism $\a$ of order $5$ such that $f_i^\a=f_{i+1}$ for all $i \in \ZZ_5$.
Then $G \in \{G_{2^4}, G_{2^8},G_{2^{12}}\}$.
 \end{lemma}
\begin{proof} Since $\a$ is a fixed-point-free automorphism of order $5$, we have $|G|\equiv 1(\mod 5)$ and hence $|G|$ is a power of $2^4$, that is, $|G|=2^{4m}$ for some positive integer $m$.
 If $N$ is a proper characteristic subgroup of $G$, then $\a$ induces a fixed-point-free automorphism of $G/N$, and $|G/N|$ is a power of $2^4$. Let $s$ be the  nilpotency class of $G$ and let \[1=\Z_{0}<\Z_1<\cdots<\Z_{s-1}<\Z_s=G \text{ and } 1=G^{(s)}<G^{(s-1)}<\cdots<G^{(1)}<G^{(0)}=G\]
be  the upper central series and lower central series of $G$, respectively. Then $\Z_1=\Z(G)$, $G^{(1)}=G'$, $\Z_i/\Z_{i-1}=\Z(G/\Z_{i-1})$ and $G^{(i)}=[G^{(i-1)},G]$ for $1\leq i\leq s$. By \cite[Section 3, Chapter 2]{DG}, $G^{(r)} \leq \Z_{s-r}$ for every $0 \leq r \leq s$. Noting that all $\Z_i$ are characteristic, we obtain that $|G|=2^{4m}$ with $m\geq s$.

By assumption, $f_0, f_1,f_2,f_3,f_4\in G$ have the following relations:
\begin{equation}\label{eq:f2}
 f_j^2=1 \text{ for all } j \in \ZZ_5, \  f_0=(f_1f_2f_3f_4)^{-1}
\end{equation}
Now we set
\begin{equation}\label{eq:edc}
 e_i=[f_i,f_{i+1}], \ d_i=[e_i,f_{i+2}], \  c_i=[d_i,f_{i+3}] \text{ for all } i \in \ZZ_5,
\end{equation}
Since $f_i^\a=f_{i+1}$, we have $f_i=f_1^{\a^{i-1}}$ and hence $e_i=e_1^{\a^{i-1}},\ d_i=d_1^{\a^{i-1}}$ and $c_i=c_1^{\a^{i-1}}$, for all $i\in\mz_5$.  By Proposition \ref{fixed-point-free-atuo}, we have
\begin{equation}\label{eq:edc0}
 e_0=(e_1e_2e_3e_4)^{-1}, \ d_0=(d_1d_2d_3d_4)^{-1},\ c_0=(c_1c_2c_3c_4)^{-1}.
\end{equation}

Assume $s=1$. Then $\Z(G)=G$ and $G$ is abelain. Since $|G|=2^{4m}$ and $G$ is generated by involutions $f_1,f_2,f_3$ and $f_4$, we obtain that $G\cong \mz_2^4$ and hence $G=G_{2^4}$.

Assume $s=2$. Then $e_i \in G'=G^{(1)} \leq \Z_{1}=\Z(G)$ for all $i\in\mz_5$, and so
\begin{equation}\label{eq:eq4}
[e_i,e_j]=[e_i,f_j]=1 \text{ for all } i,j \in \ZZ_5.
\end{equation}

By Eqs.(\ref{eq:f2})-(\ref{eq:eq4}), all elements $f_i$ and $e_i$ in $G$ satisfy the relations in
$G_{2^8}$, and therefore, $G$ is a quotient group of $G_{2^8}$, because these elements $f_i$ and $e_i$, $i\in\mz_5$, generate $G$. Since $|G|=2^{4m}$ with $m\geq s=2$ and $|G_{2^8}|=2^8$, we have $G=G_{2^8}$.

 Assume $s=3$. Then $d_i=[e_i,f_{i+2}]=[f_i,f_{i+1},f_{i+2}]  \in G^{(2)} \leq \Z_{1}=\Z(G)$, and so
 \begin{equation}\label{eq:eq5}
[d_i,d_j]=[d_i,e_j]=[d_i,f_j]=1 \text{ for all } i,j \in \ZZ_5.
\end{equation}
By the relations in Eqs.(\ref{eq:f2})-(\ref{eq:edc0}) and Eq.(\ref{eq:eq5}), $G$ is a quotient group of $G_{2^{12}}$, and since $|G|=2^{4m}$ with $m\geq s=3$, we have $G=G_{2^{12}}$.

To finish the proof, we only need to show that $s\leq 3$. We argue by contradiction and we suppose that $s\geq 4$.

Suppose $s=4$.  Then $G^{(3)} \leq \Z_{1}=\Z(G)$. Note that $[d_i,f_{j}]=[f_i,f_{i+1},f_{i+2},f_j] \in G^{(3)} \text{ for all } i,j \in \ZZ_5$, and in particular, $c_i=[d_i,f_{i+3}]\in G^{(3)} $. Then $G$ is a quotient group of the following group $H$:
\begin{eqnarray*}
 &H=&\langle c_i,d_i,e_i,f_i: i \in \ZZ_5 \ \lvert\   f_0=(f_1f_2f_3f_4)^{-1},~e_0=(e_1e_2e_3e_4)^{-1}, ~d_0=(d_1d_2d_3d_4)^{-1}, \\
 && c_0=(c_1c_2c_3c_4)^{-1},~e_i=[f_i,f_{i+1}],~ d_i=[e_i,f_{i+2}],~c_{i}=[d_i,f_{i+3}], \\
 && f_i^2=[c_i,c_j]=[c_i,d_j]=[c_i,e_j]=[c_i,f_j]=[d_i,f_j,f_{k}]=1 \text{ \rm for all } i,j,k \in \ZZ_5.
 \end{eqnarray*}
This is impossible because $|G|=2^{4m}$ with $m\geq s=4$ and $|H|=2^{12}$ by  {\sc Magma}~\cite{Magma}.

Suppose $s\geq 5$. Then $G/\Z_{s-4}=\langle f_0\Z_{s-4},f_1\Z_{s-4},f_2\Z_{s-4},f_3\Z_{s-4},
f_4\Z_{s-4} \rangle$ has nilpotent class $4$, and $\a$ induces a fixed-point-free automorphism of $G/\Z_{s-4}$, which is impossible by the above paragraph.  \end{proof}

For each $n\in \{2^4,2^8,2^{12}\}$,  define
\begin{equation}\label{eq:eq6}
\cG_n=\Cay(G_n,S), \text{ where } S=\{f_1,f_2,f_3,f_4,(f_1f_2f_3f_4)^{-1} \}.
\end{equation}
Now we are ready to prove the first part of Theorem~\ref{th:val5}.

 \medskip

\begin{proof}[Proof of Theorem\ref{th:val5}~{\rm (1)}]   To prove the sufficiency, let $n=2^4,2^8$ or $2^{12}$. By using  {\sc Magma}~\cite{Magma}, $G_n=\langle S\rangle$ has a fixed-point-free automorphism $\a$ of order $5$ such that $f_i^\a=f_{i+1}$ for all $i\in\mz_5$, where  $f_0= (f_1f_2f_3f_4)^{-1}$. Then $\a\in \Aut(G_n,S)$ and $R(G_n)\rtimes\langle \a\rangle$ is a solvable  arc-transitive group of automorphisms of $\cG_n$.

To prove the necessity, let $\Ga$ be a connected pentavalent symmetric graph which admits a quasi-semiregular automorphism and  a solvable arc-transitive group of automorphism.  By  Theorem~\ref{th:valp}~(1), $\Ga=\Cay(M,S)$, where $M$ is a $2$-group admitting a fixed-point-free automorphism $\beta$ of order $5$, and $S$ is an orbit of $\langle \beta \rangle$ consisting of involutions such that $\langle S \rangle=M$. Thus we may write $S=\{f_1,f_2,f_3,f_4,f_0\}$, where $f_i=f_1^{\beta^{i-1}}$ for all $i \in \ZZ_5$. By Proposition~\ref{fixed-point-free-atuo},  $f_0=(f_1f_2f_3f_4)^{-1}$, and by Lemma~\ref{th:M}, $M \in \{G_{2^4}, G_{2^8},G_{2^{12}}\}$ and $\Ga=\cG_{2^4},\cG_{2^8}$ or $\cG_{2^{12}}$.
\end{proof}

To prove the second part of Theorem~\ref{th:val5}, we first construct seven connected pentavalent $T$-arc-transitive coset graphs on a nonabelian simple group $T$, which contain quasi-semiregular automorphisms of order $5$.

For $n  \in \{ 36,66,126,396,1456,2016,22176\}$, define
\begin{equation}\label{eq:eq7}
\cG_n=\Cos(T, H, HxH),
\end{equation}
where $n,x,H$ and $T$ are given below:

\medskip

\begin{enumerate}[\rm (i)]
 \item $n=36$, $x=(1,3)(4,5)$, $H=\langle (1,5)(3,4), (1,5,4,6,3)\rangle \cong \D_{10} $, and $T=\langle H,x\rangle=\A_6 \leq\Sy_6$;

 \item $n=66$, $x=(1,5)(2,12)(3,9)(4,6)(7,10)(8,11)$, $H=\langle (1$, $12)(2$, $5)(3$, $11)(4$, $7)(6$, $10)(8$, $9)$, $(1,7,6,3,5)(2,11,10,4,12)\rangle \cong \D_{10}$, and $\PSL_2(11)\cong T=\langle H,x\rangle \leq\Sy_{12}$;

 \item $n=126$,  $x=(1,5)(2,6)(3,7)(4,8)$,
 $H=\langle (1,2)(3,4)$, $(1,2,3)$, $(5,6,7)$, $(7,8,9)$, $(1,2)(5,6)\rangle$ $\cong(\A_4\times\A_5)\rtimes\ZZ_2$, and $T =\langle H,x\rangle=\A_9 \leq\Sy_9$;

 \item $n=396$, $x= (1,5,2,7,10,6,11,9)(3,8)$, $H=\langle (1,11,10,2)(5,9,6,7)$, $(1$, $11$, $2$, $10$, $3)(4$, $6$, $7$, $9$, $5)\rangle$ $\cong\AGL_1(5)$, and $\M_{11}\cong T =\langle H,x\rangle \leq \Sy_{11}$;

 \item $n=1456$, $x=(1$, $22)(2$, $23)(3$, $62)(4$, $45)(5$, $31)(6$, $55)(7$, $37)(8$, $30)(9$, $38)(10$, $13)(11$, $33)(12$, $60)(14$, $21)(15$, $19)(16$, $39)(17$, $61)(18$, $44)(20$, $54)(24$, $50)(25$, $63)(26$, $47)(27$, $59)(28$, $64)(29$, $46)(34$, $35)(36$, $48)(40$, $52)(41$, $56)(42$, $53)(43$, $49)(51$, $58)(57$, $65)$, \\ $H=\langle a,b\rangle\cong\AGL_1(5)$ with $a= (1$, $41$, $15$, $8)(2$, $43$, $40$, $12)(3$, $16$, $4$, $38)(5$, $13$, $21$, $46)(6$, $58$, $25$, $11)(7$,
        $42$, $18$, $36)(9$, $62$, $39$, $45)(10$, $14$, $29$, $31)(17$, $65$, $50$, $35)(19$, $30$, $22$, $56)(20$, $28$, $27$, $26)(23$, $49$, $52$, $60)(24$, $34$, $61$, $57)(33$, $55$, $51$, $63)(37$, $53$, $44$, $48)(47$, $54$, $64$, $59)$ and
$b=(1$, $53$, $34$, $36$, $10)(2$, $13$, $49$, $65$, $64)(3$, $32$, $4$, $38$, $16)(5$, $17$, $12$, $23$, $54)(6$, $27$, $51$, $39$, $22)(7$, $44$, $14$, $61$, $41)(8$, $24$, $31$, $37$, $18)(9$, $33$, $20$, $25$, $19)(11$, $55$, $30$, $28$, $62)(15$, $29$, $42$, $57$, $48)(21$, $59$, $52$, $43$, $50)(26$, $56$, $63$, $58$, $45)(35$, $60$, $46$, $40$, $47)$,
and ${}^2B_2(8)\cong T=\langle H,x\rangle\leq \Sy_{65}$;

\item $n=2016$, $x=(1$, $20$, $18$, $21)(2$, $9)(3$, $19$, $15$, $11)(4$, $10)(5$, $17$, $16$, $6)(7$, $8$, $12$, $14)$, \\ $H=\langle a,b\rangle\cong\D_{10}$ with $a=(1$, $18)(3$, $15)(5$, $16)(6$, $17)(7$, $12)(8$, $14)(11$, $19)(20$, $21)$ and $b=(1$, $18$, $5$, $9$, $16)(2$, $15$, $7$, $12$, $3)(4$, $19$, $20$, $21$, $11)(6$, $17$, $14$, $13$, $8)$, and \\
 $\PSL_3(4)\cong T=\langle H,x\rangle\leq\Sy_{21} $;

\item $n=22176$, $x=(1, 8)(2, 21, 12, 17) (4, 15, 7, 5)(9, 22, 19, 14)(10, 13)(11, 20, 18, 16)$, \\ $H=\langle a,b\rangle\cong\AGL_{1}(5)$  with $a=(1$, $10)(2$, $4$, $12$, $7)(5$, $17$, $15$, $21)(8$, $13)(9$, $18$, $19$, $11)(14$, $16$, $22$, $20)$ and $b=(2$, $12$, $4$, $3$, $7)(5$, $14$, $13$, $22$, $15)(6$, $11$, $9$, $19$, $18)(8$, $17$, $20$, $16$, $21)$, and
$\M_{22}\cong T=\langle H,x\rangle\leq  \Sy_{22}$.
\end{enumerate}

\medskip

For each coset graph $\cG_n$,  with the help of  {\sc Magma}, we can verify that $x^2=1$, $\langle H,x\rangle=T$, $|H{:}(H\cap H^x)|=5$, and an element of order $5$ in $H$ induces a quasi-semiregular permutation on $[T{:}H]$. By  Proposition \ref{pro:cosetcons}, $\cG_n$  is a connected pentavalent $\hat{T}$-arc-transitive graph with a quasi-semiregular automorphism of order $5$. Furthermore, the automorphism group $\Aut(\cG_n)$ can also be determined by using {\sc Magma}. Then we have the following lemma.

\begin{lemma}\label{lm:examples}
For every $n  \in \{ 36,66,126,396,1456,2016,22176\}$, $\cG_n$ is a connected pentavalent symmetric graph admitting a quasi-semiregular automorphism, and for every $\cG_n$, $\Aut(\cG_n)$ and $\Aut(\cG_n)_v $ are listed in Table \ref{tb:graphs}, where  $v$ is a vertex of $ \cG_n$.
\end{lemma}

In the following, we list all non-abelian simple groups such that the centralizer of an element of order $5$ is isomorphic to $\ZZ_5$, $\ZZ_{10}$, $\ZZ_{15}$, $\ZZ_{20}$, $\ZZ_5 \times\Sy_4$ or $\ZZ_5 \times \A_4$.

\begin{lemma}\label{lm:TCpair}
Let $T$ be a non-abelian simple group, and let $T$ contain an  element $x$ of order $5$ such that $\C_T(x) \in \{ \ZZ_5,\ZZ_{10},\ZZ_{15},\ZZ_{20},\ZZ_5 \times\Sy_4, \ZZ_5 \times \A_4\}$. Then $T$ and $\C_T(x) $ are listed in  Table~\ref{tb:cT}.
\end{lemma}

 \begin{table}[htb]
\caption{Simple groups $T$ with $\C_T(x)\in \{ \ZZ_5,\ZZ_{10},\ZZ_{15},\ZZ_{20},\ZZ_5 \times \A_4, \ZZ_5 \times\Sy_4 \}$.  }\label{tb:cT}
\[
 \begin{array}{  l  l | ll | ll  } \hline
T & \C_T(x) & T & \C_T(x) & T & \C_T(x)  \\ \hline
 \A_5  &\ZZ_5& \M_{23} &\ZZ_{15} &\PSL_2(31) &\ZZ_{15}   \\
 \A_6  &\ZZ_5 & \M_{24} &\ZZ_{5}  \times \A_4 &\PSL_2(41) &\ZZ_{20}   \\
  \A_7 &\ZZ_5 &  {}^2B_2(8) &\ZZ_{5} & \PSL_4(3)& \ZZ_{20}   \\
  \A_8 &\ZZ_{15} & \PSL_2(19) &\ZZ_{10}   &     \PSL_5(2)& \ZZ_{15}   \\
  \A_9 &\ZZ_{5} \times \A_4 &  \PSL_3(4) &\ZZ_{5} & \PSU_4(3) &\ZZ_{5}   \\
  \M_{11} &\ZZ_5 &  \PSL_2(29) &\ZZ_{15}  &  \PSU_4(2) &\ZZ_{5}  \\
  \M_{12} &\ZZ_{10} &  \PSL_2(11) &\ZZ_{5}  & \PSU_5(2) &\ZZ_{15}   \\
  \M_{22} &\ZZ_5  &  \PSL_2(16) &\ZZ_{15}  &  \POmega_7(3) &\ZZ_{5} \times\Sy_4   \\
 \hline
 \end{array}
\]
\end{table}

Lemma~\ref{lm:TCpair} is proved by checking all non-abelian simple groups case by case, which is a little long. For not breaking the thread, we leave the proof of Lemma~\ref{lm:TCpair} in Section~\ref{sec:proofTC}. Now we are ready to finish the proof of Theorem~\ref{th:val5}.

\begin{proof}[Proof of Theorem\ref{th:val5}~{\rm(2)}]  Let $\Ga$ be a connected pentavalent $T$-arc-transitive graph with  a quasi-semiregular automorphism $x\in T$, where $T$ is a nonabelian simple group. Then $x$ is of order $5$, and has a unique fixed vertex, say $v\in V(\Ga)$. Clearly, $|T|/|T_v | = |V (\Ga)|  \equiv 1 \ (\mod 5)$. By Propositions~\ref{pro:qsprop}~(2) and  \ref{pro:val5stab}, $\C_T (x) \leq \C_{T_v} (x)\in \{\ZZ_5,\ZZ_{10}, \ZZ_{15}, \ZZ_{20}, \ZZ_5 \times \A_4, \ZZ_5 \times \Sy_4 \}$.  By Lemma~\ref{lm:TCpair}, the pairs $(T,\C_T(x))$ are listed in Table \ref{tb:cT}, and again by Proposition~\ref{pro:val5stab}, one may obtain all possibilities for $T_v$. Then by using {\sc Magma}~\cite{Magma} together with Proposition~\ref{pro:cosetcons}, we have $\Ga\cong K_6$ or $\Ga\cong\cG_n$ for $n=36,66,126,396,1456,2016$ or $22176$, and the detailed information are given below. Let $w$ be adjacent to $v$ in $\Ga$. Now we consider the pairs $(T,\C_T(x))$ listed in Table~\ref{tb:cT}.

Assume $(T,\C_T(x))=(\A_5,\ZZ_5)$. By Proposition \ref{pro:val5stab}, we have \[T_v \in \{ \ZZ_5,\D_{10},\AGL_1(5),\A_5, \Sy_5, \ASL_2(4), \ASigmaL_2(4)\},\]
where $|\ASL_2(4)|=2^6{\cdot}3{\cdot}5$ and $|\ASigmaL_2(4)|=2^7{\cdot}3{\cdot} 5$. Since $|T|/|T_v| = |V (\Ga)|  \equiv 1 \ (\mod 5)$, the only possibility for $ T_v $ is $ \D_{10}$. Then $|V(\Ga)|=6$ and $\Ga\cong K_6$.

Assume $(T,\C_T(x))=(\A_6,\ZZ_5)$. Similarly, since $|\A_6|=2^3{\cdot}3^2{\cdot}5$, the possibilities for  $ T_v $ are $ \D_{10}$  and $ \A_{5}$. Clearly, $\Ga\cong K_6$ if $T_v=\A_5$. Let $T_v=\D_{10}$. By {\sc Magma}~\cite{Magma}, under conjugacy, we may let $T_v=\langle (1,5)(3,4),(1,5,4,6,3)\rangle$, and $T_{vw}=\langle (1,5)(3,4) \rangle\cong \ZZ_2$. Then $\N_T(T_{vw})=\langle (1,5)(3,4),(1,3)(4,5),(2,6)(3,4)\rangle$ has order $8$, and a right transversal of $T_{vw} $ in $\N_T(T_{vw})$ is $\{1,x_1:=(1,3)(4,5),x_2:=(2,6)(3,4),x_3:=(1, 3, 5, 4)(2, 6) \}$. By Proposition~\ref{pro:cosetcons}, $\Ga \cong \Cos(T,T_v,T_vx_iT_v)$ for some $1\leq i \leq 3$. Moreover,  computation with {\sc Magma}~\cite{Magma} shows that $\Cos(T,T_v,T_vx_3T_v)\cong \Cos(T,T_v,T_vx_2T_v)\cong\Cos(T,T_v,T_vx_1T_v)$, so that $\Ga\cong\Cos(T,T_v,T_vx_1T_v)=\cG_{36}$.

Assume $(T,\C_T(x))=(\A_7,\ZZ_5)$. Then the possibilities for  $ T_v $ are $ \ASL(2,4)$  and $ \Sy_{5} $. If $T_v=\ASL(2,4)$, then under conjugacy, by {\sc Magma} we may let $T_v=\langle (1, 2, 6, 7)(3, 5)$, $(1, 4, 6, 2, 7)\rangle$ and $T_{vw}=\langle (1, 2, 6, 7)(3, 5)\rangle$, and $\N_T(T_{vw})=\langle T_{vw}, (2, 7)(3, 5)\rangle$. By Proposition~\ref{pro:cosetcons}, $\Ga \cong \Cos(T,T_v,T_v x T_v)$ with $x=(2, 7)(3, 5)$, which is impossible because $\langle T_v, x\rangle<T$ (contradicting the connectedness of $\Ga$). If $T_v=\Sy_5$, we may let $T_v=\langle (1,2,3,4,5)$, $(1,2)(3,4)$, $(1,2)(6,7) \rangle $ and $T_{vw}=\langle (1,2,3)$, $(1,2)(3,4)$, $(1,2)(6,7)\rangle$. Then $\N_T(T_{vw})= T_{vw}$, and by Proposition~\ref{pro:cosetcons}, $\Ga \cong \Cos(T,T_v,T_v 1 T_v)$ and $T=\langle T_v,1\rangle=T_v$, contradicting the connectedness of $\Ga$.

Assume $(T,\C_T(x))=(\M_{11}, \ZZ_5)$. Since $|\M_{11}|=2^4{\cdot}3^2{\cdot}5{\cdot}11$, the possibilities for  $T_v $ are $\ASL(2,4)$ and $\Sy_{5}$. If $T_v=\Sy_5$, then by {\sc Magma} we obtain $|\N_T(T_{vw}):T_{vw}|=1$, which is impossible. If $T_v=\Sy_5$, a similar computation to the case $(\A_6,\ZZ_5)$ gives rise to $\Ga=\cG_{396}$.

Assume $(T,\C_T(x))=(\M_{22}, \ZZ_{5})$. Since $|\M_{22}|=2^7{\cdot}3^2{\cdot}5{\cdot}7{\cdot}11$,  the possibilities for  $ T_v $ are $\AGL_1(5)$, $\Sy_5$ and $\ASigmaL_2(4)$. If $T_v=\Sy_5$ or $\ASigmaL_2(4)$, then $|\N_T(T_{vw}):T_{vw}|=1$, contradicting the connectedness of $\Ga$. Thus, $T_v=\AGL_1(5)$, and by using {\sc Magma}, we have $\Ga\cong\cG_{22176}$.

Assume $(T,\C_T(x))=({}^2B_2(8), \ZZ_{5})$. Since $|{}^2B_2(8)|=2^6 {\cdot} 5  {\cdot}7  {\cdot}13$, the only possibility for $ T_v $ is $\AGL_1(5) $, and then a computation with {\sc Magma} shows that $\Ga\cong \cG_{1456}$.

Assume $(T,\C_T(x))=(\PSL_3(4), \ZZ_{5})$. Since $|\PSL_3(4)|=2^6 {\cdot}3^2 {\cdot}5 {\cdot}7$, the  possibilities for  $ T_v $ are $\D_{10}$, $\A_5$ and $\ASL_2(4)$. If $T_v=\A_{5}$, by {\sc Magma}, there are seven $\A_5$ in $T$ up to conjugacy, of which one is such that $\langle T_v,\N_T(T_{vw})\rangle=\A_5$ and the others are such that $\langle T_v,\N_T(T_{vw})\rangle=\A_6$, contradicting the connectedness of $\Ga$. If $T_v=\ASL_2(4)$, then $|\N_T(T_{vw}):T_{vw}|=1$, which is also impossible. If $T_v=\D_{10}$, then by using {\sc Magma} we have $\Ga\cong \cG_{2016}$.

Assume $(T,\C_T(x))=(\PSL_2(11), \ZZ_{5})$. Since $|\PSL_2(11)|=2^2 {\cdot}3 {\cdot}5 {\cdot} 11$, the  possibilities for  $ T_v $ are $\D_{10}$ and $\A_5$. If $T_v=\A_{5}$, then $|\N_T(T_{vw}):T_{vw}|=1$, which is impossible. If $T_v=\D_{10}$, a computation with {\sc Magma} shows that $\Ga=\cG_{66}$.

Assume $(T,\C_T(x))=(\PSU_4(2), \ZZ_{5})$.  Since $|\PSU_4(2)|=2^6 {\cdot}3^4 {\cdot}5$, the  possibilities for  $ T_v $ are $\AGL_1(5)$ and $\Sy_5$. For the former, $\langle T_v,\N_T(T_{vw})\rangle=\Sy_6$, and for the latter $\langle T_v,\N_T(T_{vw})\rangle=\Sy_6$ by using {\sc Magma}, contradicting the connectedness of $\Ga$.

Assume $(T,\C_T(x))=(\PSU_4(3), \ZZ_{5})$. Since $|\PSU_4(2)|=2^7{\cdot}3^6{\cdot}5{\cdot}7$, the  possibilities for  $ T_v $ are $\AGL_1(5)$, $\Sy_5$ and $\ASigmaL_2(4)$. If $T_v=\AGL_1(5)$, by {\sc Magma} there is no $x \in \N_T(T_{vw}) $ such that  $x^2 \in T_v$, $\langle T_v,x\rangle=T$ and $|T_v:(T_v \cap T_v^x)|=5$, contradicting Proposition~\ref{pro:cosetcons}. If $T_v=\Sy_{5}$, by {\sc Magma} we obtain  $\langle T_v,\N_T(T_{vw})\rangle=\Sy_6$, contradicting the connectedness of $\Ga$. If $T_v=\ASigmaL_2(4)$, a computation with {\sc Magma} shows that $\PSU_4(3)$ has no such subgroup isomorphic to $\ASigmaL_2(4)$, a contradiction.

Assume $(T,\C_T(x))=(\M_{12}, \ZZ_{10})$.  By Proposition~\ref{pro:val5stab}, we have
\[ T_v \in\{\ZZ_2\times \D_{10}, \ZZ_2\times\AGL(1,5)\}.\]
Since $|\M_{12}|=2^6{\cdot}3^3{\cdot}5{\cdot}11$, the only possibility for $ T_v $ is $\ZZ_2 \times\AGL_1(5)$. By using {\sc Magma}, we have $|\N_T(T_{vw}):T_{vw}|=4$, and there is no $x\in \N_T(T_{vw})$ such that $x^2 \in T_v$, $\langle T_v,x\rangle=T$ and $|T_v:(T_v \cap T_v^x)|=5$, contradicting Proposition~\ref{pro:cosetcons}.

Assume $(T,\C_T(x))=(\PSL_2(19), \ZZ_{10})$. Since $|\PSL_2(19)|=2^2{\cdot}3{\cdot}5{\cdot}11$, we have $ T_v =\ZZ_2 \times \D_{10}$, and by using {\sc Magma}, we obtain $|\N_T(T_{vw}):T_{vw}|=3$ and there is no $x \in \N_T(T_{vw}) $ such that  $x^2 \in T_v$, $\langle T_v,x\rangle=T$ and $|T_v:(T_v \cap T_v^x)|=5$, contradicting Proposition~\ref{pro:cosetcons}.

Assume $(T,\C_T(x))=(\A_8,\ZZ_{15})$. By Proposition~\ref{pro:val5stab}, we have \[T_v\in \{ \AGL_2(4),\AGammaL_2(4)\},\]
 where  $|\AGL_2(4)|=2^6{\cdot}3^2{\cdot}5$  and  $|\AGammaL_2(4)|=2^7{\cdot}3^2{\cdot}5 $.
Noting that $|\A_8|=2^6 {\cdot} 3^2 {\cdot} 5 {\cdot} 7$, we have either $|V(\Ga)|= |T:T_v|\not\equiv 1 (\mod 5)$ or $|T_v|$ is not a divisor of $|T|$, a contradiction.

Assume $(T,\C_T(x))=(\M_{23}, \ZZ_{15})$. Since $|\M_{23}|=2^7{\cdot}3^2 {\cdot}5{\cdot}7 {\cdot}11{\cdot}23$, the only possibility for  $ T_v $ is $\AGammaL_2(4) $, and by using {\sc Magma}, we obtain $|\N_T(T_{vw}):T_{vw}|=1$, a contradiction.

Assume $(T,\C_T(x))=(\PSL_2(29), \ZZ_{15})$, $(\PSL_2(16), \ZZ_{15})$ or $(\PSL_2(31), \ZZ_{15})$. Note that $|\PSL_2(29)|=2^2 {\cdot} 3 {\cdot}  5 {\cdot}  7 {\cdot} 29$, $|\PSL_2(16)|=2^4{\cdot}  3{\cdot}  5{\cdot}  17$ and $|\PSL_2(31)|=2^5{\cdot} 3{\cdot} 5{\cdot} 31$. Then $|T_v|$ is not a divisor of $|T|$, a contradiction.

Assume $(T,\C_T(x))=(\PSU_5(2), \ZZ_{15})$. Since $|\PSU_5(2)|=2^{10}{\cdot}3^5{\cdot}5{\cdot}11$, the only possibility for  $ T_v $ is $\AGammaL_2(4)$, and by using {\sc Magma}, $\PSU_5(2)$ has no subgroup isomorphic to $\AGammaL_2(4)$, a contradiction.

Assume $(T,\C_T(x))=(\PSL_5(2), \ZZ_{15})$. Since $|\PSL_5(2)|=2^{10}{\cdot}3^2{\cdot}5{\cdot}7{\cdot} 31$, the only possibility for  $ T_v $ is $\AGammaL_2(4)$, and by {\sc Magma},  $|\N_T(T_{vw}):T_{vw}|=1$, a contradiction.

Assume $(T,\C_T(x))=(\PSL_2(41), \ZZ_{20})$. By Proposition~\ref{pro:val5stab}, $T_v=\ZZ_4 \times \AGL_1(5)$, which is impossible because $|T|=2^3{\cdot}3{\cdot}5{\cdot}5{\cdot}41$ and $|T_v|=2^4{\cdot}5$.

Assume $(T,\C_T(x))=(\PSL_4(3), \ZZ_{20})$. By Proposition~\ref{pro:val5stab}, $T_v=\ZZ_4 \times \AGL_1(5)$, which is impossible because $\PSL_4(3)$ has no subgroup isomorphic to $ \ZZ_4 \times \AGL_1(5)$ by {\sc Magma}.

Assume $(T,\C_T(x))=(\A_9, \ZZ_5\times \A_{4})$. By Proposition~\ref{pro:val5stab}, we have
\[ T_v\in \{ \A_4 \times \A_5,(\A_4 \times \A_5)\rtimes\ZZ_2,2^6\rtimes\GammaL_2(4)\}, \]
  where $|\A_4 \times \A_5|=2^4{\cdot}3^2{\cdot}5$,   $|,(\A_4 \times \A_5)\rtimes\ZZ_2|=2^5{\cdot}3^2{\cdot}5$ and  $|2^6\rtimes\GammaL_2(4)|=2^9{\cdot}3^2{\cdot}5$.
Since $|\A_9|=2^6{\cdot}3^4{\cdot}5{\cdot}7$, the only possibility for  $ T_v $ is $(\A_4 \times \A_5)\rtimes \ZZ_2$. By {\sc Magma}, under conjugacy, we may let $T_v= \langle (1,2)(3,4), (1,2,3), (5,6,7), (7,8,9), (1,2)(5,6) \rangle$ and $T_{vw}=\langle (1,2)(3,4)$, $(1,2,3)$, $(5,6,7)$, $(5,6)(7,8)$, $(1,2)(5,6)\rangle$. Then $\N_T(T_{vw})=T_{vw}\cup T_{vw}x$ with $x=(1$, $5)(2$, $6)(3$, $7)(4$, $8)\rangle$. Therefore, $\Ga\cong \Cos(T,T_v,T_v xT_v)=\cG_{126}$.

Assume $(T,\C_T(x))=(\M_{24}, \ZZ_{5}\times \A_4)$. Since $|\M_{24}|=2^{10}{\cdot}3^3{\cdot}5{\cdot}7{\cdot}11{\cdot}23$, the possibilities for $ T_v $ are $(\A_4 \times \A_5)\rtimes \ZZ_2$ or $2^6 \rtimes  \GammaL_2 (4)$. If $T_v=(\A_4 \times \A_5)\rtimes \ZZ_2$ then $|\N_T(T_{vw}):T_{vw}|=1$, and if $T_v=2^6 \rtimes  \GammaL_2 (4)$ then $|\N_T(T_{vw}):T_{vw}|=2$ and $\langle \N_T(T_{vw}),T_v\rangle<T $,
of which both are impossible.

Assume $(T,\C_T(x))=(\POmega_7(3), \ZZ_{5} \times \Sy_4)$. By Proposition~\ref{pro:val5stab}, $ T_v\cong \Sy_4 \times \Sy_5$. By using {\sc Magma}, $T$ has two conjugacy classes of subgroups isomorphic to $\Sy_4 \times \Sy_5$. For one class, $\langle T_v,\N_T(T_{vw})\rangle<T$, and for the other class, there is no $x \in \N_T(T_{vw}) $ such that  $x^2 \in T_v$, $\langle T_v,x\rangle=T$ and $|T_v:(T_v \cap T_v^x)|=5$, of which both are impossible.
\end{proof}

\section{Proof of Lemma \ref{lm:TCpair}}\label{sec:proofTC}

In this section, we prove Lemma \ref{lm:TCpair}. Our proof is based on the classification of finite nonabelian simple groups. It is well known that non-abelian simple groups consist of the alternating  groups $\A_n$ for $n\geq 5$, the $26$  simple sporadic groups, and simple groups of Lie type (see \cite[Definition 2.2.8]{CFSG}). For a list of nonabelian simple
groups and the isomorphisms among them, the reader may refer to Atlas \cite[Chapters 1-3]{Atlas}
 or Kleidman and Liebeck's book~\cite[Tables 5.1.A-5.1.C and Theorem 5.1.2]{K-Lie}. For a group $G$ and $g\in G$, we denote by $o(g)$ the order of $g$ in $G$.
Let $T$ be a non-abelian simple group and we always make the next hypothesis throughout this section.

\vspace{0.2cm}
\noindent{\bf Hypothesis}
$x\in T$, $o(x)=5$, and $\C_T(x) \in \{ \ZZ_5,\ZZ_{10},\ZZ_{15},\ZZ_{20},\ZZ_5 \times\Sy_4, \ZZ_5 \times \A_4\}$.

From Hypothesis,
$|\C_T(x)| \in \{5,10,15,20,60,120 \}$. Suppose that $|T|_5=5^s$ for some $s\geq 2$. Let $P$ be  a Sylow $5$-subgroup of $T$ containing $x$. Then $|P|=5^s$, and by \cite[Chapter 1: Theorem 2.11]{DG}, $\Z(P)\not=1$.
If $x\in \Z(P)$ then $P\leq \C_T(x)$ and if $x\not\in \Z(P)$ then $\Z(P)\langle x\rangle \leq \C_T(x)$, implying $5^2\mid |\C_T(x)|$, a contradiction. It follows that \[|T|_5=5.\] This fact plays an important role in this section.

We process by considering the non-abelian simple groups $T$ case by case, of which the first is the alternating groups and simple sporadic  groups.

\begin{lemma}\label{lm:ansporadic}
 Let $T$ be an alternating group $\A_n$ with $n\geq 5$ or a  simple sporadic group. Under Hypothesis, we have $(T,\C_T(x))=(\A_5,\ZZ_5),(\A_6,\ZZ_5),(\A_7,\ZZ_5)$, $(\A_8,\mz_{15})$, $(\A_9,\ZZ_5\times\A_4)$, $(\M_{11},\ZZ_5)$, $(\M_{12},\ZZ_{10})$, $(\M_{22},\ZZ_5)$, $(\M_{23},\ZZ_{15})$ or $(\M_{24},\ZZ_5\times\A_4)$.
\end{lemma}
\begin{proof}
Assume that $T=\A_n (n\geq 5)$ is an alternating  group. Since $|T|_5=5$, we have $n\leq 9$ and then $x$ have type $5^1 1^m$ in its distinct cycle decomposition with $n=5+m$, that is, $x$ is a product of one 5-cycle and $m$ $1$-cycles (fixed points). It follows that $\C_{T}(x)=(\ZZ_5 \times \Sy_m) \cap \A_{n}$, and hence $(T,\C_T(x))=(\A_5,\ZZ_5),(\A_6,\ZZ_5),(\A_7,\ZZ_5)$, $(\A_8,\ZZ_{15})$ or $(\A_9,\ZZ_5\times\A_4)$.

Assume that $T $ is one of the $26$  simple sporadic groups. By checking the order of $T$ in \cite[Table 5.1.C]{K-Lie}, we have $T\in \{ \M_{11},\M_{12},\M_{22},\M_{23},\M_{24},J_{1},J_{3},J_{4},O'N\}$ as $|T|_5=5$.
Suppose that $T=J_{1}$, $J_{3}$, $J_{4}$ or $O'N$. By Atlas \cite{Atlas}, $|\C_T(x)|=30,30,6720,180$, respectively, which contradict Hypothesis. Thus, $T\in \{ \M_{11},\M_{12},\M_{22},\M_{23},\M_{24}\}$.
Let $T=\M_{24}$. By Atlas~\cite[pp.94]{Atlas}, $T$ has only one conjugacy class of elements of order $5$, denoted by 5A, and their centralizers have order $60$. By a calculation with {\sc Magma}~\cite{Magma}, we have $\C_T(x)=\mz_5\times\A_4$. Let $T=\M_{11}$, $\M_{12}$, $\M_{22}$ or $\M_{23}$. Then a similar argument to $\M_{24}$ gives rise to $\C_T(x)=\ZZ_5$, $\ZZ_{10}$, $\ZZ_5$ or $\ZZ_{15}$, respectively. It follows that $(T,\C_T(x))= (\M_{11},\ZZ_5)$, $(\M_{12},\ZZ_{10})$, $(\M_{22},\ZZ_5)$, $(\M_{23},\ZZ_{15})$ or $(\M_{24},\ZZ_5\times\A_4)$.
\end{proof}

The following is a simple observation, which is very useful in the proofs of the following two lemmas, and its proof is straightforward.

\vspace{0.2cm}
\noindent{\bf Observation}
 Let $q$ be a power of a prime $p$ with $p\not=5$. Then the following hold:
 \begin{enumerate}[\rm (1)]
 \item  For every positive integer $s$, $q^2\equiv \pm 1 (\mod 5)$ and $q^{4s}\equiv 1 (\mod 5)$;
 \item Either $5\nmid (q^2-1)(q^6-1)$ or $5^2\mid (q^2-1)(q^6-1)$;
 \item If $q\equiv 2,3 (\mod 5)$ then $5\nmid  (q^3+1)(q-1) $ and $5^2\mid  (q^6+1)(q^4-1) $.
\end{enumerate}

Now we consider the simple exceptional  groups of Lie type.

 \begin{lemma}\label{lm:excptional}
Let $T$ be a simple exceptional group of Lie type in characteristic $p$. Under Hypothesis, we have $(T,\C_T(x))=(^2B_2(8),\ZZ_5)$.
 \end{lemma}
 \begin{proof} Let $q$ be a $p$-power. Then $T \in\{ G_2(q)$, $F_4(q)$,  ${}^2B_2(q) (q=2^{2m+1}\geq 8)$, ${}^2G_2(q)(q=3^{2m+1}\geq 27)$, ${}^2F_4(q)(q=2^{2m+1}\geq 8)$, ${}^2F_4(2)'$, ${}^3D_4(q )$, $E_6(q),$ ${}^2E_6(q)$, $E_7(q), E_8(q)\}$. By Hypothesis, $x\in T$ has order $5$ and $|T|_5=5$. Thus $T \neq {}^2F_4(2)'$ as $5 \nmid |{}^2F_4(2)'|$.

Let $p=5$. By checking the order of $T$ in \cite[Table 5.1.B]{K-Lie}, we have $q^2\mid |T|$ and hence $|T|_5\neq 5$, a contradiction.

Let $p\neq 5$. Again by \cite[Table 5.1.B]{K-Lie} and Observation~(1), if $T \in \{ F_4(q)$, $E_6(q)$, $E_7(q)$, $E_8(q)$, ${}^2E_6(q)\}$ then $5^2\mid |T|$, a contradiction. Furthermore, since $5\nmid (q^8+q^4+1)$, by Observation~(2) we have that if $T\in \{G_2(q),{}^3D_4(q )\}$ then $5\nmid |T|$ or $5^2\mid |T|$, a contradiction. By Observation~(3), if $T={}^2G_2(q)$ then $q=3^{2m+1} \equiv 2,3 (\mod 5)$ and  $5\nmid |T|$, and if $T={}^2F_4(q)$ then $q=2^{2m+1}  \equiv 2,3 (\mod 5)$ and $5^2\mid |T|$, of which both are impossible.
Thus, $T={}^2B_2(q)$ with $q=2^{2m+1}(m\geq 1) $ and $|T|=q^2(q^2+1)(q-1)$. Note that $q^2+1=(2^{2m+1}+2^{m+1}+1)(2^{2m+1}-2^{m+1}+1)$. Then it is easy to see that $|T|_5=(q^2+1)_5=(2^{2m+1}+2^{m+1}+1)_5$ or $(2^{2m+1}-2^{m+1}+1)_5$.
By \cite[Proposition 16 and Theorem 15]{S1964}, $\C_T(x)$ is abelian and $|\C_T(x)|=2^{2m+1}+2^{m+1}+1 $ or $2^{2m+1}-2^{m+1}+1 $. Since $\C_T(x)$ is abelian, $|\C_T(x)| \in \{5,10,15,20\} $ by Hypothesis. It follows that $m=1$  and $|\C_T(x)|= 2^{2m+1}-2^{m+1}+1 =5$, that is, $(T,\C_T(x))=({}^2B_2(8),\ZZ_5)$.
\end{proof}

At last, we consider the  simple classical groups.

\begin{lemma}\label{lm:classical}
Let $T$ be a  simple classical group of Lie type in characteristic $p$, but not alternating. Then $(T,\C_T(x))= (\PSL_2(11), \ZZ_{5})$, $(\PSL_2(19),\ZZ_{10})$, $(\PSL_2(29), \ZZ_{15})$,  $(\PSL_2(16),\ZZ_{15})$, $(\PSL_2(31)$, $\ZZ_{15})$, $(\PSL_2(41)$, $\ZZ_{20})$, $(\PSL_3(4), \ZZ_{5})$,  $(\PSL_4(3), \ZZ_{20})$, $(\PSL_5(2)$, $\ZZ_{15})$,
 $(\PSU_4(2), \ZZ_{5})$, $(\PSU_4(3),\ZZ_{5})$, $(\PSU_5(2),\ZZ_{15})$, or $(\POmega_7(3),\ZZ_{5} \times\Sy_4)$.
 \end{lemma}
\begin{proof} We follow \cite[Definition 2.5.10]{CFSG} to denote by $ \ID(T)$ the group of inner-diagonal automorphisms of $T$. Then $T \unlhd \ID(T)$, where $T$ is identified with its inner automorphism group. Moreover,
 $|\ID(T)/T|$ can be found  in \cite[Theorem 2.5.12 (c)]{CFSG}.

Let $q$ be a power of $p$. Following \cite{B-G}, let $G=\ID(T)$; let   $\mathbb{N}$ be the set of  positive integers and $\mathbb{N}_{0}=\mathbb{N} \cup \{ 0\} $; let $\Phi(5,q)=\mbox{min}\{d \in \mathbb{N}\ |\ 5 \div (q^d-1) \}$.
If $p\neq 5$, then $x$ is called  \emph{semisimple}, and by \cite[Theorem 3.1.12]{B-G}, $x^T=x^G$ and hence $|\C_T(x)|= |\C_G(x)||T|/|G|$.

By Hypothesis, $|T|_5=5$, and by \cite[Proposition 2.9.1]{K-Lie}, $\PSL_2(4)\cong\PSL_2(5)\cong \A_5$, $\PSL_2(9) \cong \A_6$ and $\PSL_4(2) \cong \A_8$. Since $T$ is not alternating, $T \neq \PSL_2(4)$, $\PSL_2(5)$, $\PSL_2(9)$ or $\PSL_4(2)$.
We finish the proof by considering the simple  classical  groups case by case in the following Claims~1-4.

\medskip

\noindent{\bf Claim 1.} Let $T=\PSL_n(q)$ with $n\geq 2$ and $(n,q) \neq (2,2),(2,3)$. Then we have $(T,\C_T(x))=(\PSL_2(11), \ZZ_{5})$, $(\PSL_2(19),\ZZ_{10})$, $(\PSL_2(29), \ZZ_{15})$,  $(\PSL_2(16)$, $\ZZ_{15})$, $(\PSL_2(31)$, $\ZZ_{15})$, $(\PSL_2(41)$, $\ZZ_{20})$, $(\PSL_3(4)$, $\ZZ_{5})$,  $(\PSL_4(3), \ZZ_{20})$ or $(\PSL_5(2), \ZZ_{15})$.

In this case, $|T|= (n,q-1)^{-1}q^{n(n-1)/2}\prod_{i=2}^{n}(q^i-1)$, $G=\PGL_n(q)$ and $|G/T|=(n,q-1)$.  Since $|T|_5=5$ and $T \neq \PSL_2(5)$, we have $p\neq 5$ and so $x$ is semisimple. By Observation, $q^4 \equiv 1 (\mod 5)$, and since $|T|_5=5$, we have $n\leq 7$.

Let $n=2$. By Huppert \cite[Chapter 2: Theorems 8.3-8.5]{H1967} $\C_T(x) \cong \ZZ_{(q-1)/(2,q-1)}$ or $\ZZ_{(q+1)/(2,q-1)}$. By Hypothesis, $\C_T(x) \in \{ \ZZ_5,\ZZ_{10},\ZZ_{15},\ZZ_{20},\ZZ_5 \times\Sy_4, \ZZ_5 \times \A_4\}$, and hence we have $(T,\C_T(x))= (\PSL_2(11), \ZZ_{5})$, $(\PSL_2(19)$, $\ZZ_{10})$, $(\PSL_2(29), \ZZ_{15})$,  $(\PSL_2(16)$, $\ZZ_{15})$, $(\PSL_2(31)$, $\ZZ_{15})$ or $(\PSL_2(41)$, $\ZZ_{20})$.

Let $n=3$. If $q \equiv 2,3(\mod 5)$ then $|T|_5\not=1$, and if $q \equiv 1(\mod 5)$ then $5^2 \mid |T|$, which contradict $|T|_5=5$. Thus, $q\equiv 4(\mod 5)$ and so $\Phi(5,q)=2$.
By \cite[Proposition 3.2.1]{B-G}, $r$ is $5$, the order of the semisimple element $x$, and therefore  only the second row in  \cite[Table B.3]{B-G} occurs because $i=\Phi(5,q)=2$:
 \begin{equation}\label{eq:psl}
 \begin{split}
 \textstyle |\C_G(x)|=&  \textstyle{(q-1)^{-1}|\GL_e(q)|\prod_{j=1}^{t}|\GL_{a_j}(q^i)|}  ~\text{,~where~} e,a_j \in \mathbb{N}_0, i=\Phi(5,q), \textstyle{t=\frac{5-1}{i}}  \\
 & \text{such~that~} \textstyle{n=e+i\sum_{j=1}^{t}a_j } \text{~and~} (a_1,\cdots,a_t) \neq (0,\cdots,0).
\end{split}
 \end{equation}
Then $t=2$, and since $3=n=e+2(a_1+a_2)$ and $(a_1,a_2) \neq (0,0)$, we have $e=1$ and $(a_1,a_2)=(1,0)$ or $(0,1)$. It follows that $|\C_G(x)|=(q-1)^{-1}|\GL_1(q)||\GL_1(q^2)|=q^2-1$ and so  $\C_T(x)|= |\C_G(x)||T|/|G|=(q^2-1)/(3,q-1)$.
By Hypothesis, $|\C_T(x)| \in \{5,10,15,20,60,120 \}$, and then $q=4$, $|\C_T(x)|=5$ and   $(T,\C_T(x))=(\PSL_3(4),\ZZ_5)$.

Let $4\leq n \leq 7$. If $q \equiv 1,4(\mod 5)$ then $5^2 \mid |T|$, a contradiction. Thus, $q \equiv 2,3(\mod 5)$ and $\Phi(5,q)=4$.
By \cite[Table B.3]{B-G}, $|\C_G(x)|$ also satisfies Eq.~(\ref{eq:psl}).  Then $i=4$, $t=1$, $n=e+4a_1=e+4$ ($n\leq 7$) and $|\C_G(x)|=(q-1)^{-1}|\GL_e(q)||\GL_1(q^4)|$.
If $e=2$ or $3$, then $|\C_T(x)|=\frac{|T|}{|G|}|\C_G(x)|=\frac{|\GL_e(q)||\GL_1(q^4)|}{(q-1)(4+e,q-1)}=\frac{q(q^2-1)(q^4-1)}{(6,q-1)} \text{~or~}\frac{q^3(q^3-1)(q^2-1)(q^4-1)}{(7,q-1)}$,
respectively. It is easy to check
$|\C_T(x)| \notin \{5,10,15,20,60,120 \}$, and hence $n=4$ or $5$.

Assume $n=4$. Then $e =0$ and so
\[ |\C_T(x)|=\frac{|T|}{|G|}|\C_G(x)|=\frac{|\GL_1(q^4)|}{(q-1)(4,q-1)}=\frac{q^4-1}{(q-1)(4,q-1)}\in \{5,10,15,20,60,120 \}.\]
Recall that $T \neq \PSL_4(2)$. Since $|\C_T(x)|\in \{5,10,15,20,60,120 \}$, we have $q=3$ and $T=\PSL_4(3)$. By {\sc Magma}, $\C_T(x)=\ZZ_{20}$, that is, $(T,\C_T(x))=(\PSL_4(3),\ZZ_{20})$.

Assume $n=5$. Then $e =1$ and
\[  |\C_T(x)|=\frac{|T|}{|G|}|\C_G(x)|=\frac{|\GL_1(q)||\GL_1(q^4)|}{(q-1)(5,q-1)}=\frac{q^4-1}{(5,q-1)}\in \{5,10,15,20,60,120 \}.\]
 Since $|\C_T(x)|\in \{5,10,15,20,60,120 \}$, we have $q=2$ and $(T,\C_T(x))=(\PSL_5(2), \ZZ_{15})$.

\smallskip

\noindent{\bf Claim 2.} Let $T=\PSU_n(q)$ with $n \geq 3$ and $(n,q) \neq (3,2)$. Then $(T,\C_T(x))=(\PSU_4(2),\ZZ_{5})$, $(\PSU_4(3),\ZZ_{5})$ or $(\PSU_5(2),\ZZ_{15})$. \label{psu}

In this case, $|T|=  (n,q+1)^{-1}q^{n(n-1)/2}\prod_{i=2}^{n}(q^i-(-1)^i)$, $G=\PGU_n(q)$ and $|G/T|=(n,q+1)$.
 Since $|T|_5=5$, $p\neq 5$ and $x$ is semisimple. By Observation~(1), $n\leq 7$.

Let $n=3$. Since $|T|_5=1$, we have  $q \equiv 1(\mod 5)$ and so $\Phi(5,q)=1$. By \cite[Table B.4]{B-G} and \cite[Proposition 3.3.2(i)]{B-G}, we have
 \begin{equation}\label{eq:psu}
 \begin{split}
 \textstyle |\C_G(x)|&  =\textstyle{(q+1)^{-1}|\GU_e(q)|\prod_{j=1}^{s/2}|\GL_{a_j}(q^{2b})|}  ~\text{,~where~} e,a_j \in \mathbb{N}_0, i=\Phi(5,q),   \\
  & \textstyle{s=\frac{5-1}{b}}, b=i(i\text{~odd}) \text{~or~} i/2 (i\text{~even}), \text{~such~that~} \textstyle{n=e+2b\sum_{j=1}^{s}a_j } \\ &\text{and~} (a_1,\cdots,a_t) \neq (0,\cdots,0).
\end{split}
 \end{equation}
Thus, $b=1$, $e=1$, $s/2=2$, $(a_1,a_2)=(1,0)$ or $(0,1)$. It follows that   $|\C_G(x)|=(q+1)^{-1}|\GU_1(q)||\GL_1(q^2)|=q^2-1$, and $|\C_T(x)|=|\C_G(x)||T|/|G|=(q^2-1)/(3,q+1)\in \{5,10,15,20,60,120 \}$, forcing $q=4$, which contradicts $q \equiv 1(\mod 5)$.

Let $4\leq n \leq 7$. If $q \equiv 1,4(\mod 5)$, then $5^2 \mid |T|$, a contradiction. Therefore $q \equiv 2,3(\mod 5)$ and so $\Phi(5,q)=4$.
By \cite[Table B.4]{B-G} and \cite[Proposition 3.3.2(i)]{B-G}, $|\C_G(x)|$ also satisfies Eq.~(\ref{eq:psu}), and since $s/2=1$, $|\C_G(x)|=(q+1)^{-1}|\GU_e(q)||\GL_{1} (q^4)|$, where $e \in N_0$ and $n=e+4 $.
If $e=2$ or $3$, then
$|\C_T(x)|=\frac{|T|}{|G|}|\C_G(x)|=\frac{|\GU_e(q)||\GL_1(q^4)|}{(q+1)(4+e,q+1)}=\frac{q(q^2-1)(q^4-1)}{(6,q+1)} \text{~or~}\frac{q^3(q^3+1)(q^2-1)(q^4-1)}{(7,q+1)}$,
respectively. It is easy to check
$|\C_T(x)| \notin \{5$, $10$, $15$, $20$, $60$, $120\}$, and hence $n=4$ or $5$.

Assume $n=4$. Then $e =0$ and so
\[ |\C_T(x)|=\frac{|T|}{|G|}|\C_G(x)|=\frac{|\GL_1(q^4)|}{(q+1)(4,q+1)}=\frac{q^4-1}{(q+1)(4,q+1)}\in \{5,10,15,20,60,120 \}.\]
It follows that $q=2$ and $3$, and $|\C_T(x)|=15$ and $20$, respectively. By {\sc Magma},  $(T,\C_T(x))=(\PSU_4(2),\ZZ_{15})$ or $(\PSU_4(3),\ZZ_{20})$.

Assume $n=5$. Then $e =1$ and so
\[  |\C_T(x)|=\frac{|T|}{|G|}|\C_G(x)|=\frac{|\GU_1(q)||\GL_1(q^4)|}{(q+1)(5,q+1)}=\frac{q^4-1}{(5,q+1)}\in \{5,10,15,20,60,120 \}.\]
It follows that $q=2$ and $|\C_T(x)|=15$, that is, $(T,\C_T(x))=(\PSU_5(2),\ZZ_{15})$.

%
%

 \smallskip

\noindent{\bf Claim 3.} Let $T=\PSp_{n}(q)$ with $n\geq 4$ even and $(n,q) \neq (4,2)  $.  Then   $(T,\C_T(x))=(\PSp_4(3),\ZZ_5)\cong (\PSU_4(2),\ZZ_5) $.  \label{psp}

If $(n,q)=(4,3)$, then $T=\PSp_4(3) \cong \PSU_4(2)$ by \cite[Proposition 2.9.1]{K-Lie}), which is considered in Claim~2. Thus, we may suppose $(n,q) \neq (4,3)$ and will drive a contradiction.

In this case,  $|T|= (2,q-1)^{-1}q^{n^2/4}\prod_{i=1}^{n/2}(q^{2i}-1)$, $G=\PGSp_n(q)$ and $|G/T|=(2,q-1)$.  Since $|T|_5=5$, we have $p\neq 5$ and $x$ is semisimple. By Observation, $n \in \{4,6\}$ and $q \equiv 2,3 (\mod 5)$. Then $\Phi(5,q)=4$. By \cite[Table B.7]{B-G} and \cite[Proposition 3.4.3 (ii)]{B-G},
 \begin{equation}\label{eq:psp}
 \begin{split}
 \textstyle |\C_G(x)|&  =\textstyle{ |\Sp_e(q)|\prod_{j=1}^{t}|\GU_{a_j}(q^{i/2})|}  ~\text{~where~} e,a_j \in \mathbb{N}_0, i=\Phi(5,q),   \textstyle{t=\frac{5-1}{i}}\\
  & \text{~such~that~} \textstyle{n=e+i\sum_{j=1}^{t}a_j} \text{~and~} (a_1,\cdots,a_t) \neq (0,\cdots,0).
\end{split}
 \end{equation}
Suppose $n=4$. Then $e=0$, $t=1$ and
\[ |\C_T(x)|=\frac{|T|}{|G|}|\C_G(x)|=\frac{|\GU_{1} (q^{2})}{(2,q-1)}=\frac{ q^2+1 }{(2,q-1)}\in \{5,10,15,20,60,120 \}.\]
This is impossible because $q \equiv 2,3 (\mod 5)$ and $q \notin \{2,3\}$.

Suppose $n=6$. Then $e=2$, $t=1$ and
\[ |\C_T(x)|=\frac{|T|}{|G|}|\C_G(x)|=\frac{|\Sp_2(q)||\GU_{1} (q^{2})}{(2,q-1)}=\frac{ q(q^2-1)(q^2+1) }{(2,q-1)}\in \{5,10,15,20,60,120 \}.\]
It follows that $q=3$ and $|\C_T(x)|=120$. However, by {\sc Magma}, $\C_{T}(x)\cong\ZZ_5 \times \SL_2(3) \not\cong \ZZ_5 \times \Sy_4$ for $T=\PSp_6(3)$, contradicting Hypothesis.

\smallskip

\noindent{\bf Claim 4.} Let $T=\POmega^{\epsilon}_n(q)$ with $n \geq 7$ and $\epsilon \in \{ \circ,+,-\}$, where $nq$ is odd if $\epsilon=\circ$ and $n$ is even if $\epsilon=+$ or $-$. Then  $(T,\C_T(x))=(\POmega_7(3),\ZZ_5 \times\Sy_4)$. \label{PO7}

Note that
\[
\begin{split}
&|\POmega_{n}(q)^\circ|=|\POmega_{n}(q)|=\textstyle{2^{-1}q^{ (n-1)^2/4}\prod_{i=1}^{(n-1)/2} (q^{2i}-1)} \text{~with~} nq \text{~odd,~and} \\
&|\POmega_{n}^{\pm}(q)| =\textstyle{(4,q^{n/2} \mp 1)^{-1}q^{n(n-2)/4}(q^{n/2} \mp 1)\prod_{i=1}^{n/2-1}(q^{2i}-1)} \text{~with ~} n \text{~even}.
\end{split}
 \]
 Since $|T|_5=5$, we have  $p \neq 5$ and $x$ is semisimple, and by Observation, $T=\POmega_7(q)^\circ$ or $\POmega_8^{-}(q)$, where  $q \equiv 2,3(\mod 5)$.   Then $\Phi(5,q)=4$. By \cite[Table B.12]{B-G} and \cite[Proposition 3.5.4 (ii)]{B-G}, we have
 \begin{equation}\label{eq:po}
 \begin{split}
 \textstyle |\C_G(x)|&  =\textstyle{2^{-\delta} |\O^{\epsilon'}_e(q)|\prod_{j=1}^{t}|\GU_{a_j}(q^{i/2})|}  ~\text{,~where~} e,a_j \in \mathbb{N}_0, i=\Phi(5,q),   \textstyle{t=\frac{5-1}{i}}\\
  & \text{such~that~} \textstyle{n=e+i\sum_{j=1}^{t}a_j},\ (a_1,\cdots,a_t) \neq (0,\cdots,0), \\
  & \delta= 1 \text{~if~} e \neq 0  \text{~and~} \delta= 0 \text{~otherwise~}, \text{~and~}\epsilon' = \epsilon \text{~if and only if either~}  n   \\
  & \text{is~odd,} \text{~or~} \text{there~is~an~even~number~of~odd~terms~}a_j.\\
\end{split}
 \end{equation}
Note that the above notation $\O^{\epsilon'}_e(q)$ in \cite[Section 2.5]{B-G} has the same meaning as the nation $\GO^{\epsilon'}_e(q)$ in Atlas \cite[Chapter 2: Section 4]{Atlas}.

Assume $T=\POmega_7(q)^\epsilon$ with $\epsilon=\circ$ and $q \equiv 2,3(\mod 5)$. Then $i=4$, $t=1$, $e=3$, $a_1=1$, $\delta=1$ and $\epsilon'=\epsilon$. By \cite[Theorem 2.5.12 (c)]{CFSG}),  $|G|/|T|=(2,q-1)=2$, and so
\[ |\C_T(x)|=\frac{|T|}{|G|}|\C_G(x)|=\frac{2^{-1}|\O_3(q)||\GU_1(q^2)|}{2}=\frac{q(q^2-1)(q^2+1)}{2} \in \{5,10,15,20,60,120 \}.\]
It follows that $q=3$ and $|\C_T(x)|=120$, and by {\sc Magma}, $\C_T(x) \cong \ZZ_5 \times\Sy_4$, that is, $(T,\C_T(x))=(\POmega_7(3),\ZZ_5 \times\Sy_4)$.

Assume $T=\POmega_8^{\epsilon}(q)$ with $\epsilon=-$ and $q \equiv 2,3(\mod 5)$. Then  $i=4$, $t=1$,
$e=4$, $a_1=1$, $\delta=1$ and $\epsilon'=+$. By \cite[Theorem 2.5.12 (c)]{CFSG}, $|G|/|T|=(2,q-1)$, and so
\[ |\C_T(x)|=\frac{2^{-1}|\O_4^{+}(q)||\GU_1(q^2)|}{(2,q-1)}=\frac{q^2(q^2-1)^2(q^2+1)}{(2,q-1)}\in \{5,10,15,20,60,120 \},\]
which is impossible.
\end{proof}

\section{Proof of Theorem~\ref{th:infinite}}\label{sec:6p4}

The goal of this section is to prove Theorem~\ref{th:infinite}. Let $p$ be an odd prime. We set the following  notation in this section:
\[M=\langle a\rangle\times \langle b\rangle\times\langle c\rangle\times\langle d \rangle\cong \ZZ_p^4, \mbox{ and } \Delta=\{a,b,c,d,e\} \mbox{ where } e=a^{-1}b^{-1}c^{-1}d^{-1}.\]

Let $\Aut(M,\Delta)$ be the subgroup of $\Aut(M)$ fixing $\Delta$ setwise. Since $\langle \Delta\rangle=M$, $\Aut(M,\Delta)$ has a faithful action on $\Delta$. Recall that the map  $a\mapsto b$, $b\mapsto c$, $c\mapsto d$ and $d\mapsto e$ induces an automorphism of $\Aut(M)$ of order $5$, and hence $\Aut(M,\Delta)$ is transitive on $\Delta$. Any permutation on $\{a,b,c,d\}$ induces an automorphism of $M$, and fixes $e$. This implies that $\Aut(M,\Delta)\cong \Sy_\Delta\cong\Sy_5$.  For not making the notation too cumbersome, we identify $\Aut(M,\Delta)$ with $\Sy_{\Delta}$, and then for any $\delta\in \Sy_{\Delta}$, $\delta$ also denotes the automorphism of $M$ induced by $a\mapsto a^\delta$, $b\mapsto b^\delta$, $c\mapsto c^\delta$ and $d\mapsto d^\delta$.  We further set
\[\Delta^{-1}=\{a^{-1},b^{-1},c^{-1},d^{-1},e^{-1}\}, \ \beta=(a,b^{-1},d,c^{-1})(a^{-1},b,d^{-1},c)(e,e^{-1}),   \ \text{ and }  \ L=\langle \A_\Delta,\beta\rangle,\]
where $\beta$ is the automorphism of $M$ induced by $a\mapsto b^{-1}$, $b\mapsto d^{-1}$, $c\mapsto a^{-1}$ and $d\mapsto c^{-1}$.

Clearly, $\A_{\Delta}\leq \Sy_\Delta$ also fixes $\Delta^{-1}$ setwies, and $\beta$ interchanges $\Delta$ and $\Delta^{-1}$. Then $L$ is transitive on $\Delta\cup \Delta^{-1}$ with $\{\Delta,\Delta^{-1}\}$ as a complete imprimitive block system. It follows that $\A_\Delta^\beta$ fixes $\Delta$ setwise, and hence  $\A_\Delta^\beta\leq \Sy_\Delta$. It follows that $\A_\Delta^\beta=\A_\Delta$. Clearly, $\beta^2=(a\ d)(b\ c)\in \A_\Delta$, and therefore, $|L:\A_\Delta|=2$. If $\C_L(\A_\Delta)\not=1$, then $L\cong \A_{\Delta}\times\ZZ_2$,
contradicting that $\beta$ has order $4$. Thus, $\C_L(\A_\Delta)=1$ and $L\cong\Sy_5$. In particular, we view $L$ as a permutation group on $\Delta\cup\Delta^{-1}$, and for simplicity, we still view the subgroup $\A_{\Delta}$ of $L$ as a permutation group on $\Delta$. This means that we will write an element in $\A_{\Delta}$, like $\b^2=(a,d)(b,c)$,  as a permutation on $\Delta$, and an element
in $L \setminus \A_{\Delta}$, like $\beta=(a,b^{-1},d,c^{-1})(a^{-1},b,d^{-1},c)(e,e^{-1})$, as a permutation on $\Delta\cup\Delta^{-1}$.

Note that $L\leq \Aut(M)$. Then we may define the semiproduct group $M\rtimes L$ as following:
\[G:=M\rtimes L=\{x\delta\ |\ x\in M, \delta\in L\} \mbox{ with } x\delta=\delta x^\delta,\]
where $x^\delta$ is the image of $x$ under $\delta$. Also we write $M\leq G$ and $L\leq G$ by identifying $x\in M$ and $\delta\in L$ with $x1$ and $1\delta$ in $G$, respectively. Recall that $\Delta\subset M\leq M\A_{\Delta}\leq L$.
It is easy to see that $M\A_\Delta$ has a conjugacy class $\Delta$, and $G$ has a conjugacy class $\Delta\cup \Delta^{-1}$. (Note that $L\cong\Sy_\Delta\cong\Sy_5$, and we can also define the other semiproduct $M\rtimes \Sy_\Delta$, but it cannot produce the graph in question.) Take the elements $\alpha$,  $\gamma$ and $u$ in $G$ as following:
\[\alpha=(a,b,c,d,e),\ \  \gamma=(a,b)(c,d), \ \ u=ab^{-1}c^{-1}d,\ \ H=\langle \alpha,\beta\rangle.\]

Now we are ready to define the coset graph
\[\cG_{6p^4} =\Cos(G, H, H{\gamma^u}H).\]
By using {\sc Magma}~\cite{Magma}, $\cG_{6p^4}$ has no quasi-semiregular automorphism for $p=3$ and $5$.

To prove Theorem~\ref{th:infinite}, it is enough to prove the following theorem.

\begin{theorem}\label{insolubleauto:6p^4}
For each prime $p\geq 7$, $\Ga_{6p^4}$ is a connected pentavalent symmetric graph admitting a quasi-semiregular automorphism,  and $\Aut(\cG_{6p^4})\cong \ZZ_p^4\rtimes\Sy_5$.
\end{theorem}
\begin{proof}
It is easy to see that $\alpha^\beta=(a\ b\ c\ d\ e)^\beta=(b\ d\ a\ c\ e)=\alpha^2$ and $\beta^\gamma=((a\ b^{-1}\ d\ c^{-1})(a^{-1}\ b\ d^{-1}\ c)(e,e^{-1}))^{(a\ b)(c\ d)}=\b^{-1}$.
Then $H=\langle \alpha,\beta\rangle\cong \AGL_{1}(5)$, the Froubenius group of order $20$, and
$\langle \beta,\gamma\rangle\cong \D_{8}$, the dihedral group of order $8$. Then
 \[L=\langle\gamma,H\rangle=\langle \alpha,\beta,\gamma\rangle \text{ and } {\gamma^u}=\gamma^{ab^{-1}c^{-1}d}=\gamma(a^{-1}bcd^{-1})^\gamma (ab^{-1}c^{-1}d)=\gamma u^2.\]

Noting that $u^\beta=(ab^{-1}c^{-1}d)^\beta=a^\beta (b^{-1})^\beta (c^{-1})^\beta d^\beta=ab^{-1}c^{-1}d=u$, we derive that $\beta^{\gamma^u}=(\beta^\gamma)^{u^2}=\beta^{-1}$, which implies that $\beta\in H\cap H^{\gamma^u}$.
Suppose $H^{\gamma^u}=H$. Since $\langle \alpha\rangle$ is a normal Sylow $5$-subgroup of $H$, we have $\langle\alpha\rangle^{\gamma^u}=\langle\alpha\rangle$ and hence $\langle \alpha^\gamma\rangle=\langle\alpha^{u^{-2}}\rangle$,
but this is not true because $\alpha^{u^{-2}}=u^2\alpha u^{-2}=\alpha (u^2)^\alpha u^{-2}=\alpha a^{-4}b^2c^{-2}d^{-6}\not\in L$ (note that $L\cap M=1$). Thus, $H^{\gamma^u}\not=H$, and since $\beta\in H\cap H^{\gamma^u}$, we have $H^{\gamma^u}\cap H=\langle \beta\rangle$ and $|H:H^{\gamma^u}\cap H|=5$. Since $L\cong\Sy_5$, the core $H_L=1$, and hence $H_G=1$. By Propostion~\ref{pro:cosetcons},
 $\cG_{6p^4}$ is a $\hat{G}$-arc-transitive graph of valency $5$ and order $6p^4$, where $\hat{G}\cong  G \cong \ZZ_p^4\rtimes\Sy_5$.

 \vskip 0.2cm
\noindent{\bf Claim 1:} $M$ is a minimal normal subgroup in $M\A_\Delta$, where $M\A_\Delta\leq G$.

Suppose by contradiction that $K\leq M$ is a normal subgroup of $M\A_\Delta$ with $1\not=K\not=M$. Since $M\cong\ZZ_p^4$, we have $K\cong \ZZ_p, \ZZ_p^2$ or $\ZZ_p^3$ and $M/K\cong \ZZ_p, \ZZ_p^2$ or $\ZZ_p^3$. Let $\overline{\Delta}=\{aK,bK,cK,dK,eK\}$. Since $\langle \Delta \rangle=M$, we have
$\langle \overline{\Delta}\rangle=M/K\cong \ZZ_p, \ZZ_p^2$ or $\ZZ_p^3$.

If $aK=bK$, then $bK=cK=dK=eK$ because $\alpha=(a, b,c,d,e)$ induces an automorphism of $M/K$, which forces $ab^{-1}\in K$, $ac^{-1}\in K$, $ad^{-1}\in K$ and $ae^{-1}=a^2bcd\in K$. Then $
 a^5=(ab^{-1})(ac^{-1})(ad^{-1})(a^2bcd) \in K$. Since $p\geq 7$, we derive $a\in K$, and then  $b,c,d\in K$, forcing $K=M$, a contradiction. Thus, $aK\not=bK$, and since $L\cong\A_5$ is simple, $L$ has a faithful conjugate action on $M/K$. In particular, $L\leq \Aut(M/K)$, forcing  $M/K\not\cong\ZZ_p$. Thus, $M/K\cong \ZZ_p^2$ or $\ZZ_p^3$.

Thus, we may identify $L$ with a group of automorphism of $M/K$, that is, $(xK)^\delta=x^\delta K$ for any $x\in M$ and $\delta\in L$. In particular, $(aK)^\a=bK$, $(bK)^\a=cK$, $(cK)^\a=dK$, $(dK)^\a=eK$ and $(eK)^\a=aK$. Since $M/K=\langle \overline{\Delta}\rangle$, we obtain that if $M/K\cong \ZZ_p^2$ then $\a$ cannot fix the subgroup $\langle aK\rangle$ and hence $M/K=\langle aK\rangle\times\langle bK\rangle$. Similarly, one may easily prove that if $M/K\cong \ZZ_p^3$ then $M/K=\langle aK\rangle\times\langle bK\rangle\times\langle cK\rangle $.

First let $M/K\cong \ZZ_p^2$. Then $M/K=\langle aK\rangle\times\langle bK\rangle$, and $cK=a^iKb^jK$ for some $i,j\in\ZZ_p$, where $\mz_p$ is viewed as the field of order $p$. Since $(a,b,c)\in \A_\Delta$, we have $aK=(cK)^{(a,b,c)}=b^i(a^ib^j)^jK$, forcing $ij=1$ and $i+j^2=0$ in the field $\mz_p$. Thus $j^6=1$,  and since $(6,p)=1$, we have $j=1$, and then $i=1$ and $2=0$, forcing $p=2$, a contradiction. Now let $M/K\cong \ZZ_p^3$. Then $M/K=\langle aK\rangle\times\langle bK\rangle\times\langle cK\rangle$, and $dK=a^iKb^jKc^k K$ for some $i,j,k\in \ZZ_p$. Since $(a, b,c)\in \A_\Delta$, we have $a^iKb^jKc^k K=dK=(dK)^{(a,b,c)}=b^iKc^jKa^kK$, implying that $i=j=k$, and  since $(b,c,d)\in \A_\Delta$, we have $bK=(dK)^{(b,c,d)}=(a^iKb^iKc^iK)^{(b,c,d)}=a^iKc^iK(a^ib^ic^i)^iK$, implying that $i^2+i=0$ and $i^2=1$. It follows that $i=-i^2=-1$, that is, $dK=a^{-1}Kb^{-1}Kc^{-1}K$.  Since
$(c,d,e) \in \A_\Delta$, we have $eK=(dK)^{(c,d,e)}=(a^{-1}Kb^{-1}Kc^{-1}K)^{(c,d,e)}=a^{-1}Kb^{-1}Kd^{-1}K=cK$, and since $(a,b,c,d,e)\in \A_\Delta$, we have $aK=bK=cK=dK=eK$, a contradiction. This completes the proof of Claim~1.

\vskip 0.2cm
\noindent{\bf Claim 2:} $\cG_{6p^4}$ is connected.

Write $B=\langle H,\gamma^u\rangle$. We only need to show that $B=G$. Note that $\alpha \gamma^u=\alpha \gamma u^2=(a,b,c,d,e)(a,b)(c,d)u^2=(b,d,e)u^2\in B$. Since $(u^2)^{(b,e,d)}=(a^2b^{-2}c^{-2}d^2)^{(b,e,d)}=a^2e^{-2}c^{-2}b^2$ and $(u^2)^{(b, d,e)}=a^2d^{-2}c^{-2}e^2$,  we obtain $(\alpha \gamma^u)^3=(u^2)^{(b\ e\ d)}(u^2)^{(b\ d\ e)}u^2=a^6c^{-6}\in M$. This implies that $M\cap B\not=1$ as $p\geq 7$. Note that $M\cap B\unlhd B$. Since $\gamma=u\gamma^u u^{-1}\in MB$, we have  $G=ML=\langle \gamma,H,M\rangle\leq MB$, forcing $G=MB$. It follows that $M\cap B\unlhd G$, and by Claim 1, $M\cap B=M$, implying $M\leq B$. Thus, $G=MB=B$, as required.

\vskip 0.2cm
\noindent{\bf Claim 3:} $\cG_{6p^4}$ has a quasi-semiregular automorphism.

Note that $\hat{\alpha}\in\hat{G}$ is an automorphism of  $\cG_{6p^4}$ of order $5$. We only need to show that $\hat{\alpha}$ is a quasi-semiregular automorphism of $\cG_{6p^4}$. Let $(H\delta v)^{\hat{\alpha}}=H(\delta v)\alpha=H(\delta v)$ for some $\delta v\in G=LM$ with $\delta\in L$ and $v\in M$. It suffices to show that $H(\delta v)=H$.

Note that $H(\delta v)=H(\delta v)\alpha=H(\delta v)\alpha=H\alpha^{-1}(\delta v)\alpha=H\delta^\alpha v^\alpha$.
Since $M\unlhd G$, we have $v^\a v^{-1}\in \langle H,\delta,\delta^\a\rangle\cap M\leq L\cap M=1$, forcing $v^\a=v$. By Lemma \ref{lm:p45qs}, $\a$ is a fixed-point-free automorphism of $M$. It follows that $v=1$ and $H\delta=H\delta^\alpha=H\delta\alpha$, that is, $\alpha^{\delta^{-1}} \in H$. Noting that $\langle \alpha\rangle \cong \ZZ_5$ is the unique Sylow $5$-subgroup of $H \cong \AGL_1(5)$, we have $\delta^{-1} \in \N_{L}(\langle \alpha\rangle )=H$ as $L\cong
\Sy_5$. Therefore $H(\delta v)=H$, and $\hat{\alpha}$ fixes only the vertex $H$ in $\cG_{6p^4}$, that is,  $\hat{\alpha}$ is quasi-semiregular on $V(\cG_{6p^4})$.

\vskip 0.2cm
\noindent{\bf Claim 4.}  $\Aut(\cG_{6p^4})\cong \ZZ_p^4\rtimes\Sy_5 $.

Recall that $\hat{G}\cong \ZZ_p^4\rtimes\Sy_5$. We need to prove $\Aut(\cG_{6p^4})=\hat{G}$.  Suppose by contradiction that $\hat{G}$ is a proper subgroup of $\Aut(\cG_{6p^4}) $. Then we choose a subgroup $X$ of $\Aut(\cG_{6p^4})$ such that $\hat{G}$ is maximal in $X$. Let $K=\hat{G}_X$, the largest normal subgroup of $X$ contained in $\hat{G}$.

Suppose that $K \neq 1$. Since $G=M\rtimes L\cong \ZZ_p^4\rtimes\Sy_5$, by Claim 1, $M$ is the unique minimal normal subgroup in $G$, and so is $\hat{M}$ in $\hat{G}$. It follows that $\ZZ_p^4\cong \hat{M}\leq K$. Clearly, $\hat{M}$ is characteristic in $K$, and hence normal in $X$. Note that the quotient graph $(\cG_{6p^4})_{\hat{M}}$ is $K_6$. By Proposition \ref{pro:normalquo}, $X/\hat{M}\leq \Aut(K_6) =\Sy_6$ is  arc-transitive on $K_6$,  and since
$\Sy_5\cong \hat{G}/\hat{M}<X/\hat{M} $, we have $X/\hat{M}=\Aut(K_6)=\Sy_6$.   Then Claim 1 implies  $\A_6$ has a  $4$-dimensional irreducible representation over $ \ZZ_p (p\geq 7)$, but this is impossible because dimensions of absolutely irreducible representations of $\A_6$ are $5,8,9$ and $10$ by Atlas~\cite[p.5]{Atlas}.

Thus, $K=1$ and $X$ can be viewed as a primitive permutation group of degree $|X{:}\hat{G}|$ with $\hat{G}$ as a point stabilizer. In particular, $X \leq \Sy_{|X:\hat{G}|}$.

Note that $\gamma^u\a^2\in H\gamma^uH$ and $\gamma^u\a^2=\gamma u^2\a^2=\gamma\a^2(u^2)^{\a^2}$. For convenience, write $\varepsilon=\gamma\alpha^2=(a\ b)(c\ d)(a\ b\ c\ d\ e)^2=(a\ d\ e\ b\ c)$ and $v=(u^2)^{\a^2}=(a^2b^{-2}c^{-2}d^2)^{(a\ c\ e\ b\ d)}=c^2d^{-2}e^{-2}a^2$. Then $(\gamma^u\a^2)^5=(\varepsilon v)^5=\varepsilon^5v^{\varepsilon^4}v^{\varepsilon^3}v^{\varepsilon^2}v^{\varepsilon}v=
v^{\varepsilon^4}v^{\varepsilon^3}v^{\varepsilon^2}v^{\varepsilon}v$.  Since $\varepsilon$ is conjugate to $\alpha$ by $(b,d)(c,e) \in L$, $\hat{\varepsilon}$ is quasi-semiregular on $V(\cG_{6p^4})$, and by Proposition \ref{pro:qsprop}, $\hat{\varepsilon}$ induces an fixed-point-free automorphism of $\hat{M}$. By   Proposition \ref{fixed-point-free-atuo}, we have $v^{\varepsilon^4}v^{\varepsilon^3}v^{\varepsilon^2}v^{\varepsilon}v=1$, that is, $(\gamma^u\a^2)^5=1$.
It implies that $(H$, $H\gamma^u\alpha^2$, $H(\gamma^u\alpha^2)^2$, $H(\gamma^u\alpha^2)^3$, $H(\gamma^u\alpha^2)^4,H)$ is a $5$-cycle, and by \cite[Proposition 17.2]{Biggs}, $\cG_{6p^4}$ cannot be $3$-arc-transitive.
Since $\hat{G}_v\cong H\cong\AGL_1(5)$, by \cite[Theorem 1.1]{Feng-Guo}, $X_v$ is one of the following: \[\ZZ_2 \times \AGL_1(5),\ \ZZ_4 \times \AGL_1(5),\ \Sy_5,\ (\A_4 \times \A_5)\rtimes \ZZ_2 \mbox{ and }\Sy_4 \times\Sy_5.\]

Note that $|X{:}\hat{G}|=|X_v{:}\hat{G}_v|$  and $X$ is a $\{2,3,5,p \}$-group as $|X|=|V(\Ga)||X_v|=6p^4|X_v|$.
If $X_v=\ZZ_2 \times \AGL_1(5)$, $\ZZ_4 \times \AGL_1(5)$ or $\Sy_5$, then we have $|X{:}\hat{G}|=2$, $4$, $6$  and hence $X\leq\Sy_2$, $\Sy_4$ or $\Sy_6$, which is impossible.
If $X_v=(\A_4 \times \A_5)\rtimes \ZZ_2$, then $|X{:}\hat{G}|=72$, and by  \cite[Appendix B]{DM-book}, $\soc(X)=\PSL_2(71)$ or $\A_{72}$, of which both are not a $\{2,3,5,p \}$-group, a contradiction. If $X_v=\Sy_4 \times\Sy_5$ then $|X{:}\hat{G}|=144$, and by  \cite[Appendix B]{DM-book} $\soc(X)=\PSL_2(11)^2,\M_{11}^2,\M_{12}^2,\A_{12}^2,\M_{12},\PSL_3(3)$ or $\A_{144}$. For  $\soc(X)=\A_{12}^2,\PSL_3(3),\A_{144} $, $X$ is not a $\{2,3,5,p \}$-group, a contradiction, and for  $\soc(X)=\PSL_2(11)^2,\M_{11}^2,\M_{12}^2$ or $\M_{12}$, we have $p=11$ and $|X|_{11}=11^2$, contradicting $p^4 \mid |T|$. This competes the proof.
\end{proof}

\section{Final remark}

Let $\Ga$ be a symmetric graph admitting a quasi-semiregular automorphism. When $\Ga$ has valency at most $4$, all $\Ga$'s have been completely  classified in \cite{FengHKKM} and \cite{YF}. However, it is far away from a complete classification of such graphs $\Ga$ with valency $5$. By Theorem~\ref{th:val5}, such graphs $\Ga$ with valency $5$ are classified when $\Ga$ has an arc-transitive solvable group of automorphisms, and for the case when $\Ga$ has no arc-transitive solvable group of automorphisms, by Theorems~\ref{th:valp} and \ref{th:val5}, $\Ga$ is a normal cover of $\Ga_N$ with $N$ a nilpotent group, and $\Ga_N$ is $K_6$, or one of the graphs $\cG_{n}$ with $n \in  \{36$, $66$, $126$, $396$, $1456$, $2016$, $22176\}$. Theorem~\ref{th:infinite} gives an infinitely family of such graphs that are $\mz_p^4$-normal covers of $K_6$ for all prime $p\geq 7$, but we still have nothing to say about nilpotent normal covers of $\cG_{n}$.

Based on the results given in this paper and \cite{FengHKKM,YF}, we are inclined to think that symmetric graphs with quasi-semiregular automorphisms would be rare and not easy to construct. We would like to pose the following problem.

\medskip
\noindent{\bf Problem~B} Classify or construct the normal $N$-covers of $K_6$ and $\cG_{n}$ with $n \in  \{36$, $66$, $126$, $396$, $1456$, $2016$, $22176\}$, where $N$ is a nilpotent group.



\begin{thebibliography}{99}%
\bibitem{Magma}
{\bibname W. Bosma, J. Cannon } and {\sc C. Playoust}, `The MAGMA algebra system I: The user language',   {\em J. Symbolic Comput. }23 (1997) 235-265.



\bibitem{Biggs-graphs}
 {\sc N. Biggs}, `Three remarkable graphs', {\em Canad. J. Math. }25 (1973) 397-411.

\bibitem{Biggs}
 {\sc N. Biggs}, {\em Algebraic Graph Theory, second edition} (Cambridge University Press, Cambridge, 1993).

\bibitem{B-G}
 {\sc T. C. Burness } and {\sc M. Giudici},  {\em Classical Groups, Derangements and Primes} (Australian Mathematical Society Lecture Series (25), Cambridge University Press, Cambridge, 2016).

\bibitem{CSS}
 {\sc P. Cameron, J. Sheehan } and {\sc P. Spiga},  `Semiregular automorphisms of vertex-transitive cubic graphs', {\em Eur. J. Combin. } 27 (2006) 924-930.


\bibitem{Atlas}
 {\sc J. H. Conway, R. T. Curtis, S. P. Norton, R. A. Parker } and {\sc R. A. Wilson},  {\em Atlas of Finite Groups: Maximal Subgroups and Ordinary Characters for Simple Groups} (Clarendon Press, Oxford, 1985).

\bibitem{DM-book}
  {\sc J. D. Dixon } and {\sc B. Mortimer},  {\em Permutation Groups} (Sprimger-Verlag, New York Berlin Heidelberg, 1996).

\bibitem{DMMN}
 {\sc E. Dobson, A. Malni\v{c}, D. Maru\v{s}i\v{c} } and {\sc L. A. Nowitz}, `Semiregular automorphisms of vertex-transitive graphs of certain valencies', {\em J. Comb. Theory Ser. B }97 (2007) 371-380.

\bibitem{GFR}
 {\sc J. K. Doyle, T. W. Tucker } and {\sc M. E. Watkins}, `Graphical Frobenius representations', {\em J. Algebr. Comb.} 48 (2018) 405-428.

 \bibitem{DF19}
{\sc J.-L. Du } and {\sc Y.-Q. Feng}, `Tetravalent $2$-arc-transitive Cayley graphs on non-abelian simple groups',   {\em Commun. Algebra }47 (2019)  4565-4574.


\bibitem{FengHKKM}
{\sc Y.-Q. Feng, A. Hujdurovi\'{c}, I. Kov\'{a}cs, K. Kutnar } and {\sc D. Maru\v{s}i\v{c}}, `Quasi-semiregular automorphisms of cubic and tetravalent arc-transitive graphs', {\em Appl. Math. Comput. }353 (2019) 329-337.

\bibitem{GPV}
 {\sc M. Giudici, P. Poto\v{c}nik } and {\sc G. Verret}, `Semiregular automorphisms of edge-transitive graphs', {\em J. Algebr. Comb.} 40 (2014) 961-972.


\bibitem{GV}
 {\sc M. Giudici } and {\sc G. Verret}, `Arc-transitive graphs of valency twice a prime admit a semiregular automorphism', {\em Ars Math. Contemp.} 18 (2020) 179-186.

\bibitem{GX}
{\sc  M. Giudici } and {\sc J. Xu}, `All vertex-transitive locally-quasiprimitive graphs have a semiregular automorphism', {\em J. Algebr. Comb. }25 (2007) 217-232.

\bibitem{Godsil}
{\sc  C. D. Godsil}, `On the full automorphism group of a graph', {\em Combinatorica }1 (1981) 243-256.

\bibitem{DG}
{\sc D. Gorenstein},  {\em Finite Groups} (Chelsea Publishing Company, New York, 1980).

\bibitem{CFSG}
 {\sc D. Gorenstein, R. Lyons } and {\sc R. Solomon}, {\em The Classification of the Finite Simple Groups, Number 3} (American Mathematical Society, Providence, Rhode Island, 1998).


\bibitem{Feng-Guo}
 {\sc S.-T. Guo } and {\sc Y.-Q. Feng}, `A note on pentavalent $s$-transitive graphs', {\em Discrete Math. }312 (2012) 2214-2216.


\bibitem{Higman1957}
 {\sc G. Higman}, `Groups and rings having automorphisms without non-trivial fixed elements', {\em J. Lond. Math. Soc.} 32 (1957) 321-334.

 \bibitem{Hu}
  {\sc A. Hujdurovi\'{c}}, `Quasi $m$-Cayley circulants', {\em Ars Math. Contemp. }6 (2013) 147-154.

\bibitem{H1967}
 {\sc B. Huppert},  {\em Endliche Gruppen I} (Springer-Verlag, Berlin, 1967).

\bibitem{Huppert82b}
 {\sc B. Huppert } and {\sc N. Blackburn}, {\em Finite Groups III} (Springer-Verlag, New York, 1982).


\bibitem{K-Lie}
 {\sc P. B. Kleidman } and {\sc M. W. Liebeck}, {\em The Subgroup Structure of the Finite Classical Groups} (Cambridge University Press, Cambridge, 1990).



\bibitem{KN}
 {\sc G. Korchm\'{a}ros } and {\sc G. P. Nagy},   `Graphical Frobenius representations of non-abelian groups', {\em Ars Math. Contemp. }20 (2021), accepted, https://doi.org/10.26493/1855-3974.2154.cda.


 \bibitem{KMMM}
 {\sc  K. Kutnar, A. Malni\v{c}, L. Mart\'{i}nez } and {\sc D. Maru\v{s}i\v{c}},  `Quasi $m$-Cayley strongly regular graphs', {\em J. Korean Math. Soc. }50 (2013) 1199-1211.

\bibitem{KS}
 {\sc  K. Kutnar } and {\sc P. \v{S}parl},   `Distance-transitive graphs admit semiregular automorphisms', {\em Eur. J. Combin. }31 (2010) 25-28.

\bibitem{Li}
{\sc C. H. Li}, `Semiregular automorphisms of cubic vertex transitive graphs', {\em Proc. Amer. Math. Soc.} 136 (2008) 1905-1910.

\bibitem{Lor}
 {\sc  P. Lorimer},  `Vertex-transitive graphs: symmetric graphs of prime valency', {\em J. Graph Theory }8 (1984)  55-68.

\bibitem{M81}
 {\sc  D. Maru\v{s}i\v{c}},  `On vertex symmetric digraphs', {\em Discrete Math. }36 (1981) 69-81.

\bibitem{M2018}
{\sc  D. Maru\v{s}i\v{c}}, `Semiregular automorphisms in vertex-transitive graphs of order $3p^2$', {\em Electron. J. Combin.} 25 (2018) No. 2.25, 10 pp.

\bibitem{M2017}
{\sc  D. Maru\v{s}i\v{c}}, `Semiregular automorphisms in vertex-transitive graphs with a solvable group of automorphisms', {\em Ars Math. Contemp.} 13 (2017) 461-468.


\bibitem{MS}
 {\sc  D. Maru\v{s}i\v{c} } and {\sc R. Scapellato},  `Permutation groups, vertex-transitive digraphs and semiregular automorphism', {\em Eur. J. Combin. }19 (1998) 707-712.


\bibitem{M16}
 {\sc S. M. Mirafzal}, `Some other algebraic properties of folded hypercubes', {\em Ars Combin. }124 (2016)  153-159.

\bibitem{MSV}
{\sc J. Morris, P. Spiga } and {\sc G. Verret}, `Semiregular automorphisms of cubic vertex-transitive graphs and the abelian normal quotient method', {\em Electron. J. Combin.} 22 (2015) No. 3.32, 12 pp.

 \bibitem{Ro}
 {\sc M. A. Ronan}, `Semiregular graph automorphisms and generalized quadrangles', {\em J. Comb. Theory Ser. A} 29 (1980) 319-328.

 \bibitem{Sabi}
 {\sc B. O. Sabidussi}, `Vertex-transitive graphs', {\em Monash Math. }68 (1964) 426-438.


\bibitem{Spiga1}
{\sc P. Spiga}, `On the existence of Frobenius digraphical representations', {\em Electron. J. Combin. }25 (2018)   \#P2.6.

\bibitem{Spiga2}
{\sc P. Spiga}, `On the existence of graphical Frobenius representations and their asymptotic enumeration',  {\em J. Comb. Theory Ser. B }142 (2020) 210-243.

\bibitem{S1964}
 {\sc M. Suzuki}, `On a class of doubly transitive groups',  {\em Ann. Math.  }75 (1964) 105-145.


\bibitem{Thompson1959}
{\sc J. G. Thompson}, `Finite groups with fixed-point-free automorphisms of prime order',
{\em Proc. Natl. Acad. Sci. U.S.A. }45 (1959) 578-581.

\bibitem{V}
{\sc G. Verret}, `Arc-transitive graphs of valency $8$ have a semiregular automorphism', {\em Ars Math. Contemp. }8 (2015) 29-34.


\bibitem{Xu}
{\sc J. Xu}, `Semiregular automorphisms of arc-transitive graphs with valency $pq$', {\em Eur. J. Combin. }29 (2008)  622-629.

\bibitem{YF}
{\sc F.-G. Yin} and {\sc Y.-Q. Feng}, `Symmetric graphs of valency $4$ having a quasi-semiregular automorphism', {\em Appl. Math. Comput. }399 (2021) 126014.

\end{thebibliography}
\end{document}